\documentclass[12pt]{amsart}
\usepackage{a4wide,enumerate,color}
\usepackage{graphicx}
\usepackage{mathrsfs} 
\allowdisplaybreaks

\let\pa\partial  
\let\na\nabla  
\let\eps\varepsilon  
\newcommand{\N}{{\mathbb N}}  
\newcommand{\R}{{\mathbb R}} 
\newcommand{\diver}{\operatorname{div}} 
\newcommand{\dom}{\mathscr{D}}
\newcommand{\ctot}{c_{\rm tot}}
\newcommand{\en}{{\mathcal F}}
\newcommand{\hc}{\widehat{c}}
\newcommand{\A}{{\mathcal A}}

\newtheorem{theorem}{Theorem}   
\newtheorem{lemma}[theorem]{Lemma}

\newtheorem{corollary}[theorem]{Corollary}

\begin{document}  

\title[Analysis of a single-phase flow mixture model]{Existence analysis of a 
single-phase flow mixture model with van der Waals pressure}

\author{Ansgar J\"ungel}
\address{A.J.: Institute for Analysis and Scientific Computing, Vienna University of  
	Technology, Wiedner Hauptstra\ss e 8--10, 1040 Wien, Austria}
\email{juengel@tuwien.ac.at} 

\author{Ji\v{r}\'\i\ Miky\v{s}ka}
\address{J.M.: Department of Mathematics, Faculty of Nuclear Sciences and 
Physical Engineering,
Czech Technical University in Prague, Trojanova 13, 12000 Prague 2, Czech Republic}
\email{jiri.mikyska@fjfi.cvut.cz}

\author{Nicola Zamponi}
\address{N.Z.: Institute for Analysis and Scientific Computing, Vienna University of  
	Technology, Wiedner Hauptstra\ss e 8--10, 1040 Wien, Austria}
\email{nicola.zamponi@tuwien.ac.at}

\date{\today}

\thanks{The authors have been partially supported by the bilateral Czech-Austrian
project CZ 10/2015. The first and last authors acknowledge partial support from   
the Austrian Science Fund (FWF), grants P22108, P24304, and W1245.
The second author acknowledges support from the Student Grant Agency of
the Czech Technical University in Prague, project no.\
SGS14/206/OHK4/3T/14} 

\begin{abstract}
The transport of single-phase fluid mixtures in porous media is described by
cross-diffusion equations for the mass densities. The equations
are obtained in a thermodynamic consistent way
from mass balance, Darcy's law, and the van der Waals equation
of state for mixtures. The model consists of parabolic equations with cross diffusion
with a hypocoercive diffusion operator. The global-in-time
existence of weak solutions in a bounded domain with equilibrium
boundary conditions is proved, extending the boundedness-by-entropy method.
Based on the free energy inequality, the large-time convergence of the
solution to the constant equilibrium mass density is shown.  
For the two-species model and specific diffusion matrices, an integral
inequality is proved, which reveals a minimum principle 
for the mass fractions. Without mass diffusion, 
the two-dimensional pressure is shown to converge exponentially 
fast to a constant. Numerical examples in one space dimension
illustrate this convergence.
\end{abstract}

 

\maketitle


\section{Introduction}

The transport of fluid mixtures in porous media has many important industrial
applications like oil and gas extraction, dispersion of contaminants in underground
water reservoirs, nuclear waste storage, 
and carbon sequestration. Although there are many papers on the modeling and numerical
solution of such compositional models \cite{ADF85,CHM06,CQE00,HoFi05,OlPa07,PoMi15}, 
there are no results on their mathematical analysis. In this paper, we provide
an existence analysis for a single-phase compositional model with van der Waals
pressure in an isothermal setting. From a mathematical viewpoint,  
the model consists of strongly coupled degenerate parabolic equations for the 
mass densities. The cross-diffusion coupling and the 
hypocoercive diffusion operator constitute the
main difficulty of the analysis.

Our analysis is a continuation of the program of the first and
third author to develop a theory for cross-diffusion equations possessing
an entropy (here: free energy) structure \cite{Jue15,ZaJu16}. 
The mathematical novelties are the complex 
structure of the equations and the observation 
that the solution of the binary model, for specific
diffusion matrices, satisfies an unexpected integral inequality giving rise
to a minimum principle, which generally does not hold for strongly coupled
diffusion systems.

\subsection*{Model equations}

More specifically, we consider an isothermal fluid mixture of $n$ mass densities 
$c_i(x,t)$ in a domain $\Omega\subset\R^d$ ($d\le 3$),
whose evolution is governed by the transport equations
\begin{equation}\label{1.c}
  \pa_t c_i = \diver\bigg(c_i\na p + \eps\sum_{j=1}^n D_{ij}(c)
	\na\mu_j\bigg), \quad x\in\Omega,\ t>0,\ i=1,\ldots,n,
\end{equation}
where $c=(c_1,\ldots,c_n)$. The van der Waals pressure $p=p(c)$ and the chemical 
potentials $\mu_1=\mu_1(c),\ldots,\mu_n=\mu_n(c)$ are given by
\begin{align}
  p &= \frac{\ctot}{1-\sum_{j=1}^n b_jc_j} - \sum_{i,j=1}^n a_{ij}c_ic_j, 
	\label{1.p} \\
	\mu_i &= \log c_i - \log\bigg(1-\sum_{j=1}^n b_jc_j\bigg) 
	+ \frac{b_i\ctot}{1-\sum_{j=1}^n b_jc_j} - 2\sum_{j=1}^n a_{ij}c_j.
	\label{1.mu}
\end{align}
These expressions are well defined if $(c_1(x,t),\ldots,c_n(x,t))\in\dom$ a.e., where
\begin{equation}\label{1.dom}
  \dom = \bigg\{(c_1,\ldots,c_n)\in\R^n:c_i>0\mbox{ for }i=1,\ldots,n,\ 
	\sum_{j=1}^n b_jc_j < 1\bigg\}.
\end{equation}
Here, $\ctot=\sum_{i=1}^n c_i$
is the total mass density and $\eps>0$ is a (small) parameter. The parameter
$a_{ij}=a_{ji}>0$ measures the attraction between the $i$th and $j$th
species, and $b_j>0$ is a measure of the size of the molecules. 
The diffusion matrix $D(c)=(D_{ij}(c))$ is assumed to be symmetric and positive 
semidefinite. Moreover, we suppose that the following bound holds:
\begin{equation}\label{1.D0}
  D_0|\Pi v|^2\le \sum_{i,j=1}^n D_{ij}(c)v_i v_j \le D_1|\Pi v|^2 
	\quad\mbox{for all }v\in\R^n,\ c\in\dom,
\end{equation}
for some $D_0$, $D_1>0$, where $\Pi = I - \ell\otimes\ell$ is the projection on 
the subspace of $\R^{n}$ orthogonal to $\ell := (1,\ldots,1)/\sqrt{n}$. 
A property like \eqref{1.D0} is known in the literature as {\em hypocoercivity}, 
that is, coercivity on a subspace of the considered vector space. In our case, 
the matrix $D(c)$ in \eqref{1.D0} is coercive on the orthogonal complement of 
the subspace generated by $\ell$. Bound \eqref{1.D0} is justified in the 
derivation of model \eqref{1.c}-\eqref{1.mu},
as the diffusion fluxes $J_{i} = -\eps\sum_{j=1}^{n}D_{ij}\nabla\mu_{j}$ 
must sum up to zero (see Section~\ref{sec.model}).

Equation \eqref{1.p} is the van der Waals
equation of state for mixtures, taking into account the finite size of the
molecules. Equations \eqref{1.p}-\eqref{1.mu} are derived from the 
Helmholtz free energy $\en(c)$ of the mixture; see \eqref{2.H} below. 
For details of the
modeling and the underlying assumptions, we refer to Section \ref{sec.model}. 

We impose the boundary and initial conditions
\begin{equation}\label{1.bic}
  \mu_i=0\quad\mbox{on }\pa\Omega,\ t>0, \quad c_i(\cdot,0)=c_i^0\quad\mbox{in }
	\Omega,\ i=1,\ldots,n.
\end{equation}
Note that we choose equilibrium boundary conditions. A physically more realistic
choice would be to assume that the reservoir boundary is impermeable, leading
to no-flux boundary conditions. However, conditions \eqref{1.bic} are needed
to obtain Sobolev estimates, together with the energy inequality \eqref{1.ei} below.
Numerical examples for homogeneous Neumann boundary conditions for the pressure
in case $\eps=0$ are presented in Section \ref{sec.num}.

Up to our knowledge,
there are no analytical results for system \eqref{1.c}-\eqref{1.mu} and \eqref{1.bic}. 
In the literature, Euler and Navier-Stokes models were considered with
van der Waals pressure. For instance, the existence of global classical 
solutions to the corresponding Euler equations with small initial data 
was shown in \cite{Lec11}. The existence of
traveling waves in one-dimensional Navier-Stokes with capillarity was studied
in \cite{THHC14}. Furthermore, in \cite{GMWZ06} the existence and stability of
shock fronts in the vanishing viscosity limit for Navier-Stokes equations
with van der Waals type equations of state was established.

\subsection*{Main difficulties}

A straightforward computation shows that the Gibbs-Duhem relation
$\na p=\sum_{i=1}^n c_i\na\mu_i$ holds. Therefore, \eqref{1.c} can be written as
\begin{equation}\label{1.c2}
  \pa_t c_i = \diver\sum_{j=1}^n\big((c_ic_j+\eps D_{ij}(c))\na\mu_j\big), 
	\quad i=1,\ldots,n,
\end{equation}
which is a cross-diffusion system in the so-called entropy variables $\mu_i$
\cite{Jue15}.
The matrix $(c_ic_j)\in\R^{n\times n}$ is of rank one with two eigenvalues,
a positive one and the other one equal to zero (with algebraic multiplicity $n-1$). 
Thus, if $\eps=0$, system \eqref{1.c} is not parabolic
in the sense of Petrovski \cite{Ama93}, and an existence theory for such
diffusion systems is highly nontrivial, which is the {\em first difficulty}.
The property on the eigenvalues is reflected in the 
energy estimate. Indeed, a formal computation, made rigorous below, shows that
\begin{equation}\label{1.ei}
  \frac{d}{dt}\int_\Omega \en(c)dx + \int_\Omega|\na p|^2 dx
	+ \eps \int_\Omega \na\mu : D(c)\na\mu\ dx \le 0.
\end{equation}
In case $\eps=0$ we obtain only one gradient estimate for $p$ which
is not sufficient for the analysis. There exist some results for so-called
strongly degenerate parabolic equations (for which the diffusion matrix vanishes
in some subset of positive $d$-dimensional measure) \cite{BBKT03}.
However, the techniques cannot be applied to the present problem.
Therefore, we need to assume that $\eps>0$. Then the gradient
estimates for $\Pi\mu$ and $p$ together with the boundary conditions \eqref{1.bic}
yield uniform $H^1$ bounds, which are the basis of the existence proof.
The behavior of the solutions for $\eps=0$ are studied numerically in Section
\ref{sec.num}. 

The {\em second difficulty} is the invertibility of the relation between $c$ and $\mu$,
i.e.\ to define for given $\mu$ the mass density vector $c=\Phi^{-1}(\mu)$,
where $\mu=(\mu_1,\ldots,\mu_n)$ and $\Phi:\dom\to\R^n$ is defined by \eqref{1.mu}.
A key ingredient for the proof
is the positive definiteness of the Hessian $\en''$ of the
free energy since $\pa\Phi_i/\pa c_j=\pa^2\en/\pa c_i\pa c_j$. This is
only possible under a smallness condition on the eigenvalues of $(a_{ij})$;
see Lemma \ref{lem.Hess}. This condition is not surprising since it just
means that phase separation is prohibited. The analysis of multiphase flows
requires completely different mathematical techniques;
see, e.g., \cite{Zha13} for phase transitions in Euler equations with
van der Waals pressure.

The {\em third difficulty} is the proof of $c(x,t)\in\dom$ a.e. This property 
is needed to define $p$ and $\mu_i$ through \eqref{1.p}-\eqref{1.mu},
but generally a maximum principle cannot be applied to the strongly coupled
system \eqref{1.c}. The idea is to employ the boundedness-by-entropy method as in
\cite{Jue15,ZaJu16}, i.e.\ to
work with the entropy variables $\mu=\Phi(c)$.
We show first the existence of weak solutions $\mu=(\mu_1,\ldots,\mu_n)$ to a
regularized version of \eqref{1.c2}, define $c=\Phi^{-1}(\mu)$ and perform
the de-regularization limit to obtain the existence of a weak solution
$c$ to \eqref{1.c}. Since $c(x,t)=\Phi^{-1}(\mu(x,t))\in\dom$ a.e.\ by
definition of $\Phi$, $c_i(x,t)$ turns out to be bounded. This
idea avoids the maximum principle and is the core of the boundedness-by-entropy 
method. Let us now detail our main results.

\subsection*{Global existence of solutions}

Using the boundedness-by-entropy method and the energy inequality \eqref{1.ei},
we are able to prove the global existence of bounded weak solutions.
We set $\ctot^0=\sum_{i=1}^n c_i^0$ and $\ctot^\Gamma=\sum_{i=1}^n c_i^\Gamma$.

\begin{theorem}[Existence and large-time asymptotics]\label{thm.ex1}
Let $c_i^0:\Omega\to\dom$, $i=1,\ldots,n$, 
be Lebesgue measurable and let $c^\Gamma=\Phi^{-1}(0)\in\dom$ such that
$\log \ctot^\Gamma\in L^1(\Omega)$, where 
$\Phi:\dom\to\R^n$, $\Phi(c)=\mu$, is defined by \eqref{1.mu}. 
Furthermore, let the matrices $(D_{ij})$ and $(a_{ij})$ be symmetric and
satisfy \eqref{1.D0} as well as
\begin{equation}\label{1.kappa}
  \kappa :=\frac{1}{16}\frac{\min_{i=1,\ldots,n}b_i}{\max_{i=1,\ldots,n}b_i}
	-\frac{\lambda^*}{\min_{i=1,\ldots,n}b_i} > 0,\quad
  K := 1 - \max_{1\leq i,j\leq n}b_{i}^{-1}a_{ij} > 0,
\end{equation}
respectively, where $\lambda^*$ is the maximal eigenvalue of $(a_{ij})$. Then:
\begin{enumerate}
\item[(i)] There exists a weak solution 
$c=(c_1,\ldots,c_n) : \Omega\times (0,\infty)\to\dom$ to 
\eqref{1.c}-\eqref{1.bic} satisfying the free energy inequality \eqref{1.ei} and
\begin{align*}
  & c_i-c_i^\Gamma\in L^2 (0,\infty ;H^1_0(\Omega))\cap 
	H^1 (0,\infty;H^{-1}(\Omega)), \quad i=1,\ldots,n,\\
  & |\na p|\in L^{2} (0,\infty; L^{2}(\Omega)),\quad 
	\log\ctot\in L^{\infty}(0,\infty; L^{2}(\Omega)).
\end{align*} 
\item[(ii)] There exists a constant $C>0$, depending on $\kappa$ and
$\en^*(c^0)=\en(c^0)-\en(c^\Gamma)$ such that
$$ 
  \sum_{i=1}^n\|c_i(t)-c_i^\Gamma\|_{L^2(\Omega)}^2 \le \frac{C}{1+t}\quad
	\mbox{for }t>0.
$$
\end{enumerate}
\end{theorem}

The idea of the large-time asymptotics of $c_i(t):=c_i(\cdot,t)$ 
is to exploit the energy inequality \eqref{1.ei}. 
Since it is difficult to relate the 
free energy $\en$ and its energy dissipation $-d\en/dt$, we cannot prove
an exponential decay rate although numerical experiments in \cite{MiPo16}
and Section \ref{sec.num} indicate that this is the case even when $\eps=0$. 
Instead, we show for the relative
energy $\en^*(c)=\en(c)-\en(c^\Gamma)\ge 0$ that, for some constant $C>0$ and some
nonnegative function $\Psi\in L^{1}(0,\infty)$,
$$
  \frac{d}{dt}\int_\Omega \en^*(c)dx 
  \le -\frac{C}{1+\Psi(t)}\bigg(\int_\Omega \en^*(c)dx\bigg)^2,
$$
from which we deduce that the convergence is of order $1/t$ as $t\to\infty$.
Since the free energy is strictly convex, by Lemma \ref{lem.Hess} below,
we obtain convergence in the $L^2$ norm. 

\subsection*{An integral inequality}

If $\eps=0$, we obtain only a gradient estimate for $p$.
This lack of parabolicity is compensated by the following -- surprising -- 
integral identity,
\begin{equation}\label{1.ii}
  \int_\Omega \ctot(t)f\bigg(\frac{c_1(t)}{\ctot(t)},\ldots,
	\frac{c_{n-1}(t)}{\ctot(t)}\bigg)dx 
	= \int_\Omega \ctot^{0}f\bigg(\frac{c_1^0}{\ctot^0},\ldots,
	\frac{c_{n-1}^0}{\ctot^0}\bigg)dx, \quad t>0,
\end{equation}
for arbitrary functions $f : (0,1)^{n-1}\to\R$; see the
Appendix for a formal proof.
This means that there exists a family of conserved quantities 
depending on a function of $n-1$ variables. It is unclear whether
this identity is sufficient to perform the limit $\eps\to 0$ and to
prove the existence of a solution to \eqref{1.c} with $\eps=0$.

If $\eps>0$, the integral identity \eqref{1.ii} does not hold in general.
However, for specific diffusion matrices
$D(c)$, the following inequality holds in place of \eqref{1.ii}:
\begin{equation}\label{1.ii2}
  \int_\Omega \ctot(t)f\bigg(\frac{c_1(t)}{\ctot(t)},\ldots,
	\frac{c_{n-1}(t)}{\ctot(t)}\bigg)dx
	\le \int_\Omega \ctot^{0}f\bigg(\frac{c_{1}^0}{\ctot^0},\ldots,
	\frac{c_{n-1}^0}{\ctot^0}\bigg)dx,\quad t>0,
\end{equation}
for functions $f$ specified in Theorem \ref{coro} below.
Interestingly, this implies a minimum principle for 
$c_1/\ctot,\ldots,c_{n-1}/\ctot$.
A choice of the diffusion matrix ensuring the validity of \eqref{1.ii2} is,
for given $\alpha$, $\beta\in C^0(\overline\dom)$ with  $\beta\geq 0$, 
$\alpha>0$ in $\overline{\dom}$,
\begin{equation}\label{1.D}
  D(c) = \alpha(c)(\en'')^{-1} + \beta(c)c\otimes c ,
\end{equation}
where $\en''$ is the Hessian of the free energy $\en$.
Clearly, $D(c)$ is bounded and positive definite (although not strictly) for 
$c\in\dom$. In particular, the constraint $\sum_{i=1}^n D_{ij}(c) = 0$ does not hold, 
and so the assumptions of Theorem \ref{thm.ex1} are not satisfied. 
However, with this choice of $D(c)$, equation \eqref{1.c} becomes
\begin{equation}
  \pa_{t}c_{i} = \diver\big((1 + \eps\beta(c))c_{i}\na p + \eps\alpha(c)\na c_{i}\big)\
	\quad i=1,\ldots,n, \label{1.c.simple}
\end{equation}
and the existence proof for \eqref{1.c.simple} is simpler than in the case where $D(c)$ 
satisfies \eqref{1.D0}.

\begin{corollary}[to Theorem \ref{thm.ex1}]\label{coro.ex}
Let $c_i^0:\Omega\to\dom$, $i=1,\ldots,n$, be Lebesgue measurable and let
$c^\Gamma=\Phi^{-1}(0)\in\dom$, where $\Phi:\dom\to\R^n$, $\Phi(c)=\mu$, 
is defined by \eqref{1.mu}. Furthermore, let the matrices $(D_{ij})$ and 
$(a_{ij})$ be symmetric and satisfy \eqref{1.kappa}, \eqref{1.D}.
Then there exists a weak solution $c=(c_1,\ldots,c_n): \Omega\times (0,\infty)\to\dom$
to \eqref{1.c}-\eqref{1.mu}, \eqref{1.bic}, satisfying the free energy 
inequality \eqref{1.ei} and, for $i=1,\ldots,n$,
\begin{align*}
  c_i-c_i^{\Gamma} &\in L^2(0,\infty ;H^1_{0}(\Omega))\cap H^1(0,\infty ; 
	H^{-1}(\Omega)), \\
  \na\sqrt{c_i}, \ \na p,\ \na\log c_{i} &\in L^{2}(0,\infty; L^{2}(\Omega)), \\ 
  \log c_{i} &\in L^{\infty}(0,\infty; L^{1}(\Omega)).
\end{align*}
\end{corollary}

Our second main result reads as follows.

\begin{theorem}[Integral inequality and minimum principle]\label{coro}
Let $c_i^0=c_i^\gamma$ for $i=1,\ldots,n$ on $\pa\Omega$.
Under the assumptions of Cor\-ollary \ref{coro.ex}, the solution $c$ to 
\eqref{1.c}-\eqref{1.mu}, \eqref{1.bic} constructed in Corollary \ref{coro.ex} 
satisfies \eqref{1.ii2} for all functions $f\in C^2([0,1]^{n-1})$ such that its 
Hessian $f''$ is positive semidefinite in $[0,1]^{n-1}$ and
\begin{equation}\label{1.f} 
  f\bigg(\frac{c_1^\Gamma}{\ctot^\Gamma},\ldots,
	\frac{c_{n-1}^\Gamma}{\ctot^\Gamma}	\bigg) = 0, \quad
  \bigg|f'\bigg(\frac{c_1^\Gamma}{\ctot^\Gamma},\ldots,
	\frac{c_{n-1}^\Gamma}{\ctot^\Gamma}
	\bigg)\bigg| = 0.
\end{equation}
Moreover, for any $i=1,\ldots,n$, 
$$
  \inf_{\Omega\times(0,\infty)}\frac{c_{i}}{\ctot} \geq\min\left\{
  \inf_{\Omega}\frac{c_{i}^{0}}{\ctot^{0}},\ \frac{c_{i}^{\Gamma}}{\ctot^{\Gamma}}
  \right\}.
$$
\end{theorem}

\subsection*{Exponential convergence of the pressure}

In the degenerate situation $\eps=0$, 
we are able to show an exponential decay rate for the pressure $p$, at least
for sufficiently smooth solutions whose existence is assumed. 
The key idea of the proof is to analyze the parabolic equation satisfied by $p$,
$$
  \pa_t p = \widetilde D\Delta p + |\na p|^2, \quad\mbox{where }
	\widetilde D = \sum_{i,j=1}^n c_ic_j\frac{\pa^2\en}{\pa c_i\pa c_j}.
$$
Because of the quadratic gradient term, we need a smallness assumption on $\na p$
at time $t=0$. Thus, the exponential convergence result 
holds sufficiently close to equilibrium.

\begin{theorem}[Exponential decay of the pressure]\label{thm3}
Let $\eps=0$, $d=2$, and let $c=(c_1,\ldots,c_n)$ 
be a solution to \eqref{1.c}-\eqref{1.p}
with isobaric boundary conditions $p=p^\Gamma$ on $\pa\Omega$, $t>0$, for
some constant $p^\Gamma\in\R$. Let $m:=\min\{\inf_\Omega p(c^0),p^\Gamma\}>0$.
We assume that
$$
  \na c_i\in L^4_{\rm loc}(0,\infty);L^2(\Omega)), \quad 
	\na p\in C^0([0,\infty);L^2(\Omega))\cap L^2(0,T;H^1(\Omega)),
$$
and $\sup_{\Omega\times(0,T)}\sum_{i=1}^n b_ic_i<1$ for any $T>0$. Then there
exists $K_0>0$, which depends on $\Omega$ and $d$,
such that if $\|\na p(c^0)\|_{L^2(\Omega)}\le K_0 m$, then, for some $\lambda>0$,
$$
  \|\na p(c(t))\|_{L^2(\Omega)} \le \|\na p(c^0)\|_{L^2(\Omega)} 
	e^{-\lambda t}, \quad t>0.
$$
\end{theorem}

The paper is organized as follows. Details on the modeling of the fluid
mixture are presented in Section \ref{sec.model}. Auxiliary results on the
Hessian of the free energy, the relation between $c$ and $\mu$, and the
diffusion matrix \eqref{1.D} are shown in Section \ref{sec.aux}. In
Section \ref{sec.ex1}, we prove Theorem \ref{thm.ex1} and Corollary \ref{coro.ex}, 
while the proofs of Theorems \ref{coro} and \ref{thm3} are presented
in Section \ref{sec.coro} and \ref{sec.thm3}, respectively. 
The evolution of the one-dimensional mass densities and the pressure 
are illustrated numerically in Section \ref{sec.num} for the case $\eps=0$.
Finally, identity \eqref{1.ii} is verified in the Appendix.


\section{Modeling and energy equation}\label{sec.model}

We consider the isothermal flow of $n$ chemical components in a porous 
domain $\Omega\subset\R^d$ with porosity $\varphi$.
The transport of the partial mass densities $c_i$ is governed by 
the balance equations for the mass,
$$
  \pa_t (\varphi c_i) + \diver(c_i v_i) = 0, \quad i=1,\ldots,n,
$$
where $v_i$ is the partial velocity of the $i$th species. 
In order to derive equations for the mass densities only, we impose some
simplifying assumptions. To shorten the presentation, we set all physical 
constants equal to one. Moreover, we set
$\varphi\equiv 1$ to simplify the mathematical analysis. Our results will
be also valid for (smooth) space-dependent porosities.
Introducing the diffusion fluxes by
$J_i=c_i(v_i-v)$, where $v=\sum_{i=1}^n c_iv_i/\ctot$ is the barycentric velocity
and $\ctot=\sum_{i=1}^n c_i$ denotes the total mass density, the balance equations
become
\begin{equation}\label{m.me}
  \pa_t c_i + \diver(c_i v + J_i) = 0, \quad i=1,\ldots,n.
\end{equation}

We suppose that the barycentric velocity is given by Darcy's law $v=-\na p$,
where $p$ is the fluid pressure. We refer to \cite{Whi86} for a justification
of this law. 
The second assumption is that the diffusion fluxes are driven by the gradients
of the chemical potentials $\mu_i$, i.e.\
$J_i = -\eps\sum_{j=1}^n D_{ij}\na\mu_i$ for $i=1,\ldots,n$; see, 
e.g., \cite[Section 4.3]{LMM11}. 
Here, $\eps>0$ is some number and
$D_{ij}$ are diffusion coefficients depending on $c=(c_1,\ldots,c_n)$.
According to Onsager's principle of thermodynamics, the diffusion matrix 
$(D_{ij})$ has to be symmetric and positive semidefinite; moreover, for 
consistency with the definition $J_i = c_i(v_i-v)$, it must hold that
$\sum_{i=1}^n D_{ij}=0$ for $j=1,\ldots,n$.

The equations are closed by specifying the Helmholtz free energy density
\begin{equation}\label{2.H}
  \en(c) = \sum_{i=1}^n c_i(\log c_i-1)
	- \ctot\log\bigg(1-\sum_{j=1}^n b_jc_j\bigg) - \sum_{i,j=1}^n a_{ij}c_ic_j,
\end{equation}
where $b_j$ and $a_{ij}$ are positive numbers, and $(a_{ij})$ is symmetric.
The first term in the free energy is the internal energy and the remaining two
terms are the energy contributions of the van der Waals gas 
\cite[Formula (4.3)]{Joh14}. 

The third assumption is that the fluid is
in a single state, i.e., no phase-splitting occurs. Mathematically, this means
that the free energy must be convex. This is the case
if the maximal eigenvalue of $(a_{ij})$ is sufficiently small; see Lemma 
\ref{lem.Hess}.
The single-state assumption is restrictive from a physical viewpoint. 
It may be overcome by considering the transport equations for each phase separately
and imposing suitable boundary conditions at the interface \cite[Section~1]{LMM11}.
However, this leads to free-boundary cross-diffusion problems which we are not
able to treat mathematically.
Another approach would be to consider a two-phase (or even multi-phase) 
compositional model with overlapping of different phases,
like in \cite{PoMi15}. In such a situation, a new formulation of the thermodynamic 
equilibrium based upon the minimization of the Helmholtz free energy is employed 
to describe the splitting of components among different phases.

The chemical potentials are defined in terms of the free energy by
$$
  \mu_i = \frac{\pa \en}{\pa c_i} 
	= \log c_i - \log\bigg(1-\sum_{j=1}^n b_jc_j\bigg) 
	+ \frac{b_i\ctot}{1-\sum_{j=1}^n b_jc_j} - 2\sum_{j=1}^n a_{ij}c_j,
$$
and the pressure is determined by the Gibbs-Duhem equation \cite[Formula (64)]{BoDr15}
\begin{equation}\label{m.gde}
  p = \sum_{i=1}^n c_i\mu_i - \en(c)
	= \frac{\ctot}{1-\sum_{j=1}^n b_jc_j} - \sum_{i,j=1}^n a_{ij}c_ic_j.
\end{equation}
This describes the van der Waals equation of state for mixtures,
where the parameter $a_{ij}$ is a measure of the attractive force between the
molecules of the $i$th and $j$th species, and the parameter $b_j$ is a measure
of the size of the molecules. The pressure stays finite if $\sum_{j=1}b_jc_j<1$,
which means that the mass densities are bounded. In the literature, many
modifications of the attractive term have been proposed. Examples are the so-called
Peng-Robinson and Soave-Redlich-Kwong equations; see \cite{Soa72}.

Taking the gradient of \eqref{m.gde} and observing that $\pa \en/\pa c_i=\mu_i$,
\eqref{m.gde} can be written as
\begin{equation}\label{2.nap}
  \na p = \sum_{i=1}^n c_i\na\mu_i.
\end{equation}
Therefore, we can formulate \eqref{m.me} as the cross-diffusion equations
$$
  \pa_t c_i = \diver\bigg(\sum_{j=1}^n(c_ic_j + \eps D_{ij})\na\mu_j\bigg),
	\quad i=1,\ldots,n.
$$
Multiplying this equation by $\mu_i$, summing over $i=1,\ldots,n$, observing again that
$\mu_i=\pa \en/\pa c_i$, and integrating by parts, we arrive at the energy equation
$$
  \frac{d}{dt}\int_\Omega \en(c)dx = \int_\Omega\sum_{i=1}^n \mu_i\pa_t c_i dx
	= -\int_\Omega \bigg|\sum_{i=1}^n c_i\na\mu_i\bigg|^2 dx 
	- \eps\int_\Omega \sum_{i,j=1}^n D_{ij}\na\mu_i\cdot\na\mu_j dx.
$$
Since $(D_{ij})$ is assumed to be positive definite on $\mbox{span}\{\ell\}^\perp$,
where $\ell=(1,\ldots,1)/\sqrt{n}$, and $c\cdot\ell = \ctot/\sqrt{n}$,
this gives, thanks to Lemma \ref{lem.ab}, $L^2$ estimates
for $\ctot\na\mu_i$ and, thanks to the equilibrium boundary condition and 
Poincar\'e's inequality, $H^1$ estimates for $c_i$.


\section{Auxiliary results}\label{sec.aux}

First we show a result estimating the norms of two vectors from below.

\begin{lemma}\label{lem.ab}
Let $\alpha$, $\beta\in\R^n$ be such that $|\alpha|=|\beta|=1$. Then, for any
$v\in\R^n$,
$$
  |\alpha\cdot v|^2 + |v-(\beta\cdot v)\beta|^2 \ge \frac14(\alpha\cdot\beta)^2|v|^2.
$$
\end{lemma}

The constant $1/4$ is not optimal. For instance, if $\alpha=\beta$, we have the
theorem of Pythagoras, $|\alpha\cdot v|^2 + |v-(\beta\cdot v)\beta|^2=|v|^2$.

\begin{proof}
Let $w=(\beta\cdot v)\beta$ be the projection of $v$ on $\beta$
and $w^\perp=v-(\beta\cdot v)\beta$ be the orthogonal part.
Then, clearly, $|v|^2=|w|^2 + |w^\perp|^2$. By Young's inequality
with $\delta=3/4$ and $|\alpha|=1$, we have
\begin{align*}
  |\alpha\cdot v|^2 + |v-(\beta\cdot v)\beta|^2 
	&= |\alpha\cdot(w+w^\perp)|^2 + |w^\perp|^2 \\
	&= (\alpha\cdot w)^2 + (\alpha\cdot w^\perp)^2 
	+ 2(\alpha\cdot w)(\alpha\cdot w^\perp) + |w^\perp|^2 \\
	&\ge (1-\delta)(\alpha\cdot w)^2 + (1-\delta^{-1})(\alpha\cdot w^\perp)^2 
	+ |w^\perp|^2 \\
	&= \frac14(\alpha\cdot w)^2 - \frac13(\alpha\cdot w^\perp)^2 + |w^\perp|^2 \\
	&\ge \frac14(\beta\cdot v)^2(\alpha\cdot\beta)^2
	- \frac13|w^\perp|^2 + |w^\perp|^2.
\end{align*}
We deduce from $|\alpha|=|\beta|=1$ that $(\alpha\cdot\beta)^2\le 1$, and thus,
\begin{align*}
  |\alpha\cdot v|^2 + |v-(\beta\cdot v)\beta|^2 
	&\ge \frac14(\alpha\cdot\beta)^2|w|^2 + \frac23(\alpha\cdot\beta)^2|w^\perp|^2 \\
	&\ge \frac14(\alpha\cdot\beta)^2\big(|w|^2 + |w^\perp|^2\big) 
	= \frac14(\alpha\cdot\beta)^2|v|^2,
\end{align*}
finishing the proof.
\end{proof}

\begin{lemma}[Positive definiteness of $\en''$]\label{lem.Hess}
Let $A=(a_{ij})$, defined in the pressure relation \eqref{1.p}, 
be a symmetric matrix whose maximal eigenvalue $\lambda^*\in\R$ satisfies
\eqref{1.kappa}.
Then the Hessian $\en''$ of the free energy $\en$ is positive definite, i.e.
$$
  v\cdot \en''(c)v\ge \kappa\sum_{i=1}^n \frac{v_i^2}{c_i}
	\quad\mbox{for all }c\in\dom,\ v\in\R^n, 
$$
where $\kappa>0$ is given by \eqref{1.kappa}. In particular, $\ctot\en''$ is
uniformly positive definite.
\end{lemma}

\begin{proof}
A straightforward computation shows that $(\en'')_{ij}=B_{ij}-a_{ij}$, where
$$
  B_{ij} = (b_i+b_j)\sigma + b_ib_j\ctot\sigma^{2} 
	+ \frac{\delta_{ij}}{c_i},
	\quad \sigma = \frac{1}{1-\sum_{j=1}^n b_jc_j}\ge 1.
$$
Let $v\in\R^n$. It holds that
\begin{align*}
  \sum_{i,j=1}^n B_{ij}v_iv_j
	&= 2\sigma\bigg(\sum_{i=1}^n v_i\bigg)\bigg(\sum_{j=1}^n b_jv_j\bigg)
	+ \sigma^2\ctot\bigg(\sum_{i=1}^n b_jv_j\bigg)^2
	+ \sum_{i=1}^n\frac{v_i^2}{c_i} \\
	&= \bigg(\sigma\sqrt{\ctot}\sum_{j=1}^n b_jv_j + \frac{1}{\sqrt{\ctot}}
	\sum_{i=1}^n v_i\bigg)^2 + \sum_{i=1}^n\frac{v_i^2}{c_i}
	- \frac{1}{\ctot}\bigg(\sum_{i=1}^n v_i\bigg)^2.
\end{align*}
Defining $w_i=v_i/\sqrt{c_i}$, 
$\widehat\alpha_i = \sigma\sqrt{\ctot c_i}b_i + \sqrt{c_i/\ctot}$, and 
$\beta_i=\sqrt{c_i/\ctot}$ for $i=1,\ldots,n$, the quadratic form can be rewritten
as
$$
  \sum_{i,j=1}^n B_{ij}v_iv_j = (\widehat\alpha\cdot w)^2
	+ |w-(\beta\cdot w)\beta|^2.
$$
Since $|\widehat\alpha|^2=\sum_{i=1}^n\sigma^2
(\ctot b_i^2c_i + 2\sigma b_ic_i) + 1 \ge 1$,
we may define $\alpha=\widehat\alpha/|\widehat\alpha|$, which yields
\begin{equation}\label{2.aux1}
  \sum_{i,j=1}^n B_{ij}v_iv_j \ge (\alpha\cdot w)^2	+ |w-(\beta\cdot w)\beta|^2.
\end{equation}
The norm of $\widehat\alpha$ can be estimated from above:
$$
  |\widehat\alpha|^2 \le \sigma^2\ctot\max_{j=1,\ldots,n}b_j\sum_{i=1}^n b_ic_i 
	+ 2\sigma\sum_{i=1}^n b_ic_i + 1.
$$
Since $\sum_{i=1}^n b_ic_i<1$, we have $\min_{j=1,\ldots,n}b_j\ctot<1$
or $\ctot<1/\min_{j=1,\ldots,n}b_j$, and $\sum_{i=1}^n b_ic_i<1\le\sigma$.
Therefore,
$$
  |\widehat\alpha|^2 \le \sigma^2\frac{\max_{j=1,\ldots,n}b_j}{\min_{j=1,\ldots,n}b_j}
	+ 2\sigma^2 + \sigma^2 
	= \bigg(3+\frac{\max_{j=1,\ldots,n}b_j}{\min_{j=1,\ldots,n}b_j}\bigg)\sigma^2.
$$
We infer that $\alpha\cdot\beta$ is strictly positive:
$$
  (\alpha\cdot\beta)^2 
	= |\widehat\alpha|^{-2}\bigg(1 + \sigma\sum_{i=1}^n b_ic_i\bigg)^2
	\ge \frac{(\sigma^{-1}+\sum_{i=1}^n b_ic_i)^2}{3+\frac{\max_{j=1,\ldots,n}b_j}{
	\min_{j=1,\ldots,n}b_j}}
	= \frac{1}{3+\frac{\max_{j=1,\ldots,n}b_j}{\min_{j=1,\ldots,n}b_j}}.
$$

We apply Lemma \ref{lem.ab} to \eqref{2.aux1} to obtain
$$
   4\sum_{i,j=1}^n B_{ij}v_iv_j \ge (\alpha\cdot\beta)^2|w|^2
	\ge \frac{|w|^2}{3+\frac{\max_{j=1,\ldots,n}b_j}{\min_{j=1,\ldots,n}b_j}},
$$
which, since $w_i = v_i/\sqrt{c_i}$, implies that
\begin{equation}
  4\sum_{i,j=1}^n B_{ij}v_iv_j 
	\ge \frac{\min_{j=1,\ldots,n}b_j}{3\min_{j=1,\ldots,n}b_j 
	+ \max_{j=1,\ldots,n}b_j}\sum_{i=1}^n\frac{v_i^2}{c_i}
	\ge \frac{\min_{j=1,\ldots,n}b_j}{4\max_{j=1,\ldots,n}b_j}\sum_{i=1}^n
	\frac{v_i^2}{c_i}. \label{Bvv}
\end{equation}
The relation $\ctot<1/\min_{j=1,\ldots,n}b_j$ and the definition of $\lambda^*$ 
allow us to write
$$
  \sum_{i,j=1}^n a_{ij}v_i v_j 
	\leq \ctot\sum_{i,j=1}^n a_{ij}\frac{v_i}{\sqrt{c_i}}\frac{v_j}{\sqrt{c_j}} 
	\leq \frac{\lambda^*}{\min_{j=1,\ldots,n}b_j}\sum_{i=1}^n\frac{v_i^2}{c_i}.
$$
This, together with \eqref{Bvv}, yields the desired lower bound for $\en''$.
\end{proof}

\begin{lemma}[Invertibility of $c\mapsto\mu$]\label{lem.inv}
The mapping $\Phi:\dom\to\R^n$, $\Phi(c)=\mu:=(\mu_1,\ldots,\mu_n)$ is invertible.
\end{lemma}

\begin{proof}
Since $\en''=\pa\mu/\pa c$ is positive 
definite in $\dom$, it follows that $\Phi$ is one-to-one and the image $\Phi(\dom)$
is open. We claim that $\Phi(\dom)$ is also closed. Then
$\Phi(\dom)=\R^2$, and the proof is complete.

Let $\mu^{(m)}=\Phi(c^{(m)})$, $m\in\N$, define a sequence in $\Phi(\dom)$ such that
$\mu^{(m)}\to\overline\mu$ as $m\to\infty$. The claim follows if we prove that
there exists $\overline c\in\dom$ such that $\overline\mu=\Phi(\overline c)$.
Since $c^{(m)}=(c_1^{(m)},\ldots,c_n^{(m)})\in\dom$ 
varies in a bounded subset of $\R^n$,
the theorem of Bolzano-Weierstra{\ss} implies the existence of a subsequence,
which is not relabeled, such that $c_i^{(m)}$ converges to some $\overline c_i$
as $m\to\infty$, where $\overline c_i\in\overline{\dom}$, $i=1,\ldots,n$. We assume,
by contradiction, that $\overline c=(\overline c_1,\ldots,\overline c_n)\in\pa\dom$. 
Let us distinguish two cases.

{\em Case 1:} There exists $j\in\{1,\ldots,n\}$ such that $\overline c_j=0$.
If $\sum_{i=1}^n b_i\overline c_i<1$, then \eqref{1.mu} implies that 
$\mu^{(m)}_j\to-\infty$, which contradicts the fact that $\mu^{(m)}$ is convergent. 
Thus it holds that $\sum_{i=1}^n b_i\overline c_i=1$.
This means that $\overline c_k>0$ for some $k\in\{1,\ldots,n\}$. However, 
choosing $i=k$ in \eqref{1.mu} and exploiting the relation 
$\sum_{i=1}^n b_i\overline c_i=1$ leads to $\overline\mu_i=+\infty$, contradiction.

{\em Case 2:} For all $i\in\{1,\ldots,n\}$, it holds that $\overline c_i>0$
and $\sum_{i=1}^n b_i\overline c_i=1$. Arguing as in case 1, it follows that
$\overline\mu_i=+\infty$ for all $i=1,\ldots,n$, which is absurd.

We conclude that $\overline c\in\dom$, which finishes the proof.
\end{proof}

\begin{lemma}\label{lem.B}
Let $B(c) = c\otimes c + \eps D(c)$, $\widehat{B}(c) = \hc\otimes\hc + \eps D(c)$ 
for $c\in\dom$, where $\hc = c/|c|$ and $D(c)$ satisfies \eqref{1.D0}. Then
for all $v\in\R^n$ and $c\in\dom$,
$$
  v\cdot B v \geq k_B \left( \ctot^2 |v|^2 + |\Pi v|^2\right), \quad
  v\cdot\widehat{B} v\geq k_B'|v|^2, 
$$
where $k_B>0$, $k_B'>0$ only depend on $\eps D_0$ (the constant in \eqref{1.D0}) 
and $b_1,\ldots,b_n$.
\end{lemma}

\begin{proof}
{}From \eqref{1.D0}, $|c|^2\ge \ctot^2/n$, and 
$\ctot\leq 1/\min\{b_1,\ldots,b_n\}$ it follows
that
\begin{align*}
  v\cdot B v &\geq (c\cdot v)^2 + \eps D_0 |\Pi v|^2 \\
  &\geq \ctot^2\bigg( \frac{1}{n}(\hc\cdot v)^2 
  + \frac{\eps}{2} D_0\min\{b_1,\ldots,b_n\}^2|(I-\ell\otimes\ell)v|^2\bigg)
  + \frac{\eps}{2} D_0 |(I-\ell\otimes\ell)v|^2, 
\end{align*}
Applying Lemma \ref{lem.ab} with $\alpha=\hc$ and $\beta=\ell$ to the expression 
in the brackets yields
$$
  v\cdot B v \geq \frac{1}{4}\min\left\{\frac{1}{n}, \frac{\eps}{2} 
	D_0\min\{b_1,\ldots,b_n\}^2\right\}\ctot^2 (\hc\cdot\ell)^2 |v|^2 
  + \frac{1}{2}\eps D_0|(I-\ell\otimes\ell)v|^2 .
$$
Since $|\hc\cdot\ell| = \ctot/|c|\geq 1$, this finishes the proof of the 
first inequality. The second one is proved in an analogous way.
\end{proof}


\section{Proof of Theorem \ref{thm.ex1}}\label{sec.ex1}

We consider the following time-discretized and regularized problem in $\Omega$:
\begin{align}\label{3.disc}
  \frac{c_i^{k}-c_i^{k-1}}{\tau} 
	&= \diver\bigg(\sum_{j=1}^n B_{ij}^{k}\na\mu_j^{k}\bigg) \\
  &\phantom{xx}{}+ \tau\diver\bigg( |\hc^k\cdot\na\mu^k|^2
	(\hc^k\cdot\na\mu^k)\hc^k_i 
	+ (\na\mu^k: D^k\na\mu^k)\sum_{j=1}^n D_{ij}^k\na\mu_j^k \bigg), \nonumber
\end{align}
with homogenous Dirichlet boundary conditions
\begin{equation}\label{3.bc}
  \mu_i^{k} = 0\quad\mbox{on }\pa\Omega,\ i=1,\ldots,n,
\end{equation}
where $c_i^{k-1}\in L^\infty(\Omega)$ is given, $\tau>0$,
$c^{k} = \Phi^{-1}(\mu^{k})$, $\hc^k = c^k/|c^k|$, $D^k=D(c^k)$, and
$$
  B_{ij}^{k} = c_i^{k}c_j^{k} + \eps D_{ij}(c^{k}).
$$
We write $\na\mu^k: D^k\na\mu^k 
= \sum_{i,j=1}^n D_{ij}^{k}(c)\na\mu_i^{k}\cdot\na\mu_j^{k}$.
Note that $B^{k}=(B_{ij}^{k})$ is positive definite by Lemma \ref{lem.B}.

\subsection{Existence for the time-discretized problem.} We reformulate
\eqref{3.disc}-\eqref{3.bc} as a fixed-point problem for a suitable operator.
Let $F:L^\infty(\Omega;\R^n)\times[0,1]\to L^\infty(\Omega;\R^n)$,
$F(\mu^*,\sigma)=\mu$, where $\mu=(\mu_1,\ldots,\mu_n)$ solves
\begin{align}\label{3.semi}
  \frac{\sigma}{\tau}(c_i^{*}-c_i^{k-1}) 
	&= \diver\bigg(\sum_{j=1}^n B_{ij}^{*}\na\mu_j\bigg)\\
  &\phantom{xx}{}+ \tau\diver\bigg( |\hc^*\cdot\na\mu|^2(\hc^*\cdot\na\mu)\hc^*_i 
	+ (\na\mu: D^*\na\mu)\sum_{j=1}^n D_{ij}^*\na\mu_j \bigg), \nonumber
\end{align}
where
$$
  B^*_{ij} = c^*_i c^*_j + \eps D_{ij}^*, \quad D^* = D(c^*), 
	\quad c^*=\Phi^{-1}(\mu^*).
$$
In order to solve \eqref{3.semi}, we show that the operator $\A:X\to X'$ defined by
\begin{align*}
  \langle \A(u), v \rangle &= \int_\Omega\big( \na v: B^*\na u 
	+ \tau(\hc^*\cdot\na v)\cdot(\hc^*\cdot\na u)|\hc^*\cdot\na u|^2 \\
  &\phantom{xx}{}+ \tau(\na u: D^*\na u)(\na v: D^*\na u) \big) dx
\end{align*}
with $X=W_0^{1,4}(\Omega;\R^n)$ satisfies the assumptions of Theorem 26A in
\cite{Zei90}. Since $\mu^*\in L^\infty(\Omega;\R^n)$, we have
$B^*$, $D^*\in L^\infty(\Omega;\R^{n\times n})$, and $\A $ is well defined.

{\em Strict monotonicity:} Let $u$, $v\in X$. Then
\begin{align*}
  &\langle \A (u)-\A (v),u-v\rangle = \int_\Omega\na(u-v): B^*\na(u-v) dx\\
  &\phantom{x}{}+ \tau\int_\Omega(\hc^*\cdot\na(u-v))\cdot\big( (\hc^*\cdot\na u)
	|\hc^*\cdot\na u|^2 - (\hc^*\cdot\na v)|\hc^*\cdot\na v|^2 \big) dx\\
  &\phantom{x}{}+ \tau\int_\Omega\na(u-v):\left( (\na u: D^*\na u)D^*\na u 
	- (\na v: D^*\na v)D^*\na v \right) dx =: I_1 + I_2 + I_3.
\end{align*}
The positive definiteness of $B^*$ implies that $I_1\ge 0$. We claim that
also $I_2\ge 0$. Indeed, by decomposing $\na u=\frac12\na(u+v)+\frac12\na(u-v)$
and $\na v=\frac12\na(u+v)-\frac12\na(u-v)$, we obtain
\begin{align*}
  (\hc^*\cdot\na u) |\hc^*\cdot\na u|^2 &- (\hc^*\cdot\na v)|\hc^*\cdot\na v|^2\\
  &= \frac{1}{2}\big(\hc^*\cdot\na(u+v)\big)\big( |\hc^*\cdot\na u|^2 
	- |\hc^*\cdot\na v|^2 \big) \\
  &\phantom{xx}{}+ \frac{1}{2}\big(\hc^*\cdot\na(u-v)\big)\big( 
	|\hc^*\cdot\na u|^2 + |\hc^*\cdot\na v|^2 \big),
\end{align*}
and since $(\hc^*\cdot\na(u-v))\cdot(\hc^*\cdot\na(u+v)) 
= |\hc^*\cdot\na u|^2 - |\hc^*\cdot\na v|^2$, we deduce that
$$
  I_2 = \frac{\tau}{2}\int_\Omega\left( 
  ( |\hc^*\cdot\na u|^2 - |\hc^*\cdot\na v|^2 )^2 
	+ |\hc^*\cdot\na(u-v)|^2 ( |\hc^*\cdot\na u|^2 + |\hc^*\cdot\na v|^2 )\right)dx,
$$
which means that $I_2\geq 0$. With the same technique one can prove 
that also $I_3\geq 0$. We conclude that $\A $ is monotone. If 
$\langle \A (u)-\A (v),u-v\rangle=0$, then in particular $I_1=0$, which, thanks to
the positive definiteness of $B^*$, implies that $\na u=\na v$ and $u=v$ in $X$.
Therefore, $\A $ is strictly monotone.

{\em Coercivity:} Let $u\in X$. Since $B^*$ is positive definite, we find that
\begin{align*}
  \langle \A (u), u\rangle 
	&\geq \tau\int_\Omega\left( |\hc^*\cdot\na u|^4 + (\na u: D^*\na u)^2 \right)dx\\
  &\geq \frac{\tau}{2}\int_\Omega \big(|\hc^*\cdot\na u|^2 + \na u: D^*\na u\big)^2 dx.
\end{align*}
Lemma \ref{lem.B} implies that
$|\hc^*\cdot\na u|^2 + \na u: D^*\na u\geq k_B' |\na u|^2$, so
we infer from Poincar\'e's inequality (with constant $C_P>0$) that
$$
  \frac{\langle \A (u),u\rangle}{\|u\|_X}
  \geq \frac{\tau (k_B')^2}{2\|u\|_X}\int_\Omega|\na u|^4 dx
  \geq \frac{\tau}{2}(k_B')^2 C_P\|u\|_X^3 \to\infty
$$
as $\|u\|_X\to\infty$. Thus, $\A $ is coercive.

{\em Hemicontinuity:} Let $u$, $v$, $w\in X$. The function $t\mapsto
\langle \A (u+tv),w\rangle$ is a polynomial and is, in particular, continuous.
It follows that $\A $ is hemicontinuous.

The assumptions of Theorem 26A in \cite{Zei90} are fulfilled, and we infer
the existence of a unique solution $\mu\in X$ to \eqref{3.semi}. This shows
that the operator $F$ is well defined. If $\sigma=0$, we have $F(\cdot,0)=0$
thanks to the uniqueness of the solution to \eqref{3.semi}. A uniform bound
for all fixed points to \eqref{3.semi} and $\sigma\in[0,1]$ follows from
the above coercivity estimate for $\A $.

Let us show that $F$ is continuous. Then, because of the compact embedding
$W^{1,4}(\Omega)\hookrightarrow L^\infty(\Omega)$ for $d\le 3$, $F$ is also
compact. Let $(\mu^*)^{(m)}\in L^\infty(\Omega;\R^n)$, $n\in\N$, define a
sequence converging to $\overline\mu^*$ in $L^\infty(\Omega)$ and let
$\sigma^{(m)}\subset[0,1]$ be such that $\sigma^{(m)}\to\overline\sigma$ 
as $m\to\infty$. Set $\mu^{(m)}:=F((\mu^*)^{(m)},\sigma^{(m)})$ and
$\overline\mu := F(\overline\mu^*,\overline\sigma)$. 
The claim follows if we show that
$\mu^{(m)}\to\overline\mu$ in $L^\infty(\Omega)$.
We formulate \eqref{3.semi} compactly as $\A [\mu^*](\mu)=f(\mu^*,\sigma)$,
where $f(\mu^*,\sigma)=\sigma\tau^{-1}(c^*-c^{k-1})$, putting in
evidence the dependence on $\mu^*$. By definition, $\A [(\mu^*)^{(m)}](\mu^{(m)})
= f((\mu^*)^{(m)},\sigma^{(m)})$ and $\A [\overline\mu^*](\overline\mu)
= f(\overline\mu^*,\overline\sigma)$. It follows that
\begin{align}
  \big\langle & \A [(\mu^*)^{(m)}](\mu^{(m)}) - \A [(\mu^*)^{(m)}](\overline\mu),
	\mu^{(m)}-\overline\mu\big\rangle 
	+ \langle \A [(\mu^*)^{(m)}](\overline\mu) - \A [\overline\mu^*](\overline\mu),
	\mu^{(m)}-\overline\mu\big\rangle \nonumber \\
	&= \big\langle f((\mu^*)^{(m)},\sigma^{(m)}) - f(\overline\mu^*,\overline\sigma),
	\mu^{(m)}-\overline\mu\big\rangle. \label{3.aux1}
\end{align}
Clearly, $(\mu^{(m)})$ is bounded in $W^{1,4}(\Omega)$ and, by the compact
embedding, also in $L^\infty(\Omega)$. This fact, together with the
convergences $(\mu^*)^{(m)}\to \overline\mu^*$ in $L^\infty(\Omega)$ and
$\sigma^{(m)}\to\sigma$, implies that
\begin{align*}
  \big\langle \A [(\mu^*)^{(m)}](\overline\mu) - \A [\overline\mu^*](\overline\mu),
	\mu^{(m)}-\overline\mu\big\rangle &\to 0, \\
	\big\langle f((\mu^*)^{(m)},\sigma^{(m)}) - f(\overline\mu^*,\overline\sigma),
	\mu^{(m)}-\overline\mu\big\rangle &\to 0.
\end{align*}
Consequently, by \eqref{3.aux1},
$$
  \big\langle \A [(\mu^*)^{(m)}](\mu^{(m)}) - \A [(\mu^*)^{(m)}](\overline\mu),
	\mu^{(m)}-\overline\mu\big\rangle \to 0.
$$
The previous monotonicity estimate for $\A $ shows that
$$
  \big\langle \A [(\mu^*)^{(m)}](\mu^{(m)}) - \A [(\mu^*)^{(m)}](\overline\mu),
	\mu^{(m)}-\overline\mu\big\rangle \ge
	\int_\Omega\na(\mu^{(m)}-\overline\mu)^\top (B^*)^{(m)}\na(\mu^{(m)}-\overline\mu)dx.
$$
Then we deduce from the strict positivity of $(B^*)^{(m)}$ 
and the Poincar\'e inequality
that $\mu^{(m)}\to\overline \mu$ strongly in $H^1(\Omega)$. The uniform bound
for $(\mu^{(m)})$ in $W^{1,4}(\Omega)$ implies that $\mu^{(m)}\to\overline\mu$
strongly in $W^{1,q}(\Omega)$ for any $1<q<4$. Take $q\in(3,4)$. Then
the embedding $W^{1,q}(\Omega)\hookrightarrow L^\infty(\Omega)$ is compact,
and, possibly for a subsequence, $\mu^{(m)}\to\overline\mu$ strongly in
$L^\infty(\Omega)$. By the uniqueness of the limit, the convergence holds for
the whole sequence. This shows the continuity of $F$.

We can now apply the fixed-point theorem of Leray-Schauder to conclude the 
existence of a weak solution to \eqref{3.disc}.

\subsection{Uniform estimates} 

Let $\mu^{k}\in X$ be a solution to \eqref{3.disc}.
Employing $\mu_i^k$ as a test function and summing over $i=1,\ldots,n$ gives
$$
  \frac{1}{\tau}\sum_{i=1}^n(c_i^{k}-c_i^{k-1})\mu_i^{k} dx
  + \int_\Omega\left( \na\mu^{k}: B^{k}\na\mu^{k} 
  + \tau|\hc^k\cdot\na\mu^k|^4 + \tau(\na\mu^k: D^k\na\mu^k)^2 \right)dx = 0,
$$
where $c^{k}=\Phi^{-1}(\mu^{k})$.
Since $\mu_i^{k}=\pa \en(c^{k})/\pa c_i$ and $\en(c^{k})$ is convex,
it follows that $\sum_{i=1}^n(c_i^{k}-c_i^{k-1})\mu_i^{k}\ge \en(c^{k})-\en(c^{k-1})$
and therefore,
\begin{align}\label{3.ei}
  \int_\Omega \en(c^{k})dx &+ \int_\Omega\left( \na(\mu^{k}: B^{k}\na\mu^{k} 
  + \tau|\hc^k\cdot\na\mu^k|^4 + \tau(\na\mu^k: D^k\na\mu^k)^2 \right)dx\\
  &\le \int_\Omega \en(c^{k-1})dx.\nonumber
\end{align}
Lemma \ref{lem.B} shows that
\begin{align}\label{3.p}
  \na(\mu^{k})^\top B^{k}\na\mu^{k}
	&\geq k_B\left( (\ctot^k)^2 |\na\mu^k|^2 + |\na(\Pi\mu^k)|^2 \right),\\
  |\hc^k\cdot\na\mu^k|^4 + (\na\mu^k\cdot D^k\na\mu^k)^2 &\geq \frac{1}{2}(k_B')^2 
  |\na\mu^k|^4. \label{3.p.b}
\end{align}

Let $T>0$, $\tau=T/N$ for some $N\in\N$. We introduce the piecewise constant
functions in time $\mu^{(\tau)}(x,t)=\mu^k(x)$ for $x\in\Omega$ and 
$t\in((k-1)\tau,k\tau]$, $k=1,\ldots,N$. The functions $c^{(\tau)}$
and $B^{(\tau)}$ are defined in a similar way. Furthermore, we introduce
the shift operator $\sigma_\tau\mu^{(\tau)}(x,t)=\mu^{k-1}(x)$ for $x\in\Omega$
and $t\in((k-1)\tau,k\tau]$. Then \eqref{3.disc} can be formulated as
\begin{align}\label{3.tau}
  &\frac{c^{(\tau)}-\sigma_\tau c^{(\tau)}}{\tau} 
	= \diver\big(B^{(\tau)}\na\mu^{(\tau)}\big) \\
  &\phantom{x}{}+ \tau\diver\left( |\hc^{(\tau)}\cdot\na\mu^{(\tau)}|^2
	(\hc^{(\tau)}\cdot\na\mu^{(\tau)})\cdot\hc^{(\tau)}
  + (\na\mu^{(\tau)}: D^{(\tau)}\na\mu^{(\tau)}) D^{(\tau)}\na\mu^{(\tau)} 
	\right). \nonumber
\end{align}
Now, we sum \eqref{3.ei} over $k=1,\ldots,N$ and employ \eqref{3.p} and 
\eqref{3.p.b} to obtain
\begin{align}\label{en.in.tau}
  \int_\Omega & \en(c^{(\tau)}(x,T))dx 
  + k_B\int_0^T\int_\Omega\big( (\ctot^{(\tau)})^2|\na\mu^{(\tau)}|^2 
	+ |\na\Pi\mu^{(\tau)}|^2 \big) dx dt\\
  &{}+ \frac{\tau(k_B')^2}{2}\int_0^T\int_\Omega |\na\mu^{(\tau)}|^4 dx dt 
  \leq \int_\Omega \en(c_i^0)dx.\nonumber
\end{align}
In the following, $C>0$ denotes a generic constant independent of $\tau$ and $T$, 
while $C_{T}>0$ denotes a constant depending on $T$ but not on $\tau$.
We deduce from \eqref{en.in.tau} and Poincar\`e's Lemma that
\begin{align}
\label{est.p.mu}
  \|p^{(\tau)}\|_{L^2(0,T; H^1(\Omega))} 
	+ \|\Pi\mu^{(\tau)}\|_{L^2(0,T; H^1(\Omega))} &\leq C,\\
  \label{est.tau.mu}
  \tau^{1/4}\|\mu^{(\tau)}\|_{L^{4}(0,T; W^{1,4}(\Omega))} &\leq C.
\end{align}
By Lemma \ref{lem.Hess}, the matrix $\ctot^{(\tau)}\en''(c^{(\tau)})$ 
is uniformly positive definite. Thus, the uniform bound for
$\ctot\na\mu_i$ in $L^2$ provided by \eqref{en.in.tau} implies a uniform bound for 
$$
  \frac{\pa c^{(\tau)}}{\pa x_j}
	= (\ctot^{(\tau)}\textbf{}\en'')^{-1}\ctot^{(\tau)}\frac{\pa\mu^{(\tau)}}{\pa x_j}
$$ 
for all $j=1,\ldots,n$ in $L^2(Q_T)$, where $Q_T=\Omega\times(0,T)$. 
Therefore, since $\dom$ is bounded and $c^{(\tau)}(x,t)\in\dom$,
\begin{equation}\label{est.c}
  \|c_i^{(\tau)}\|_{L^\infty(Q_T)}
	+ \|c_i^{(\tau)}\|_{L^2(0,T;H^1(\Omega))} 
	\le C, \quad i=1,\ldots,n.
\end{equation}
In particular, $B^{(\tau)}$ is uniformly
bounded in $L^\infty(Q_T)$. Using these estimates in \eqref{3.tau} shows that
\begin{align} \label{est.dc}
  \tau^{-1} \|c_i^{(\tau)}-\sigma_\tau c_i^{(\tau)}\|_{L^{4/3}(0,T;W^{1,4}(\Omega)')}
  &\leq C\big(\|\nabla p\|_{L^{4/3}(Q_T)} + \|\nabla\Pi\mu\|_{L^{4/3}(Q_T)} \\ 
	&\phantom{xx}{}+ \tau\|\nabla\mu\|_{L^{4}(Q_T)}^3\big) 
	\leq C. \nonumber
\end{align}

\subsection{The limit $\tau\to 0$} 

In view of estimates \eqref{est.c} and \eqref{est.dc}, we can apply the
Aubin-Lions lemma in the version of \cite{DrJu12}, ensuring the existence of 
a subsequence, which is not relabeled, such that, as $\tau\to 0$,
$$
  c_i^{(\tau)}\to c_i \quad\mbox{strongly in }L^4(Q_T),\
	i=1,\ldots,n.
$$
In fact, in view of the $L^\infty(Q_T)$ bound \eqref{est.c}, this convergence 
holds in $L^q(Q_T)$ for any $q<\infty$. Furthermore, we have
$$
  \tau^{-1}(c_i^{(\tau)}-\sigma_\tau c_i^{(\tau)})
	\rightharpoonup \pa_t c_i \quad\mbox{weakly in }L^{4/3}(0,T;W^{1,4}(\Omega)'),\
	i=1,\ldots,n.
$$

It holds that $c(x,t)\in\overline\dom$ for a.e.\ $(x,t)\in Q_T$. 
Let $\mu:=\Phi(c), p=p(c)\in(\R\cup\{\pm\infty\})^n$.
By \eqref{est.p.mu}, \eqref{est.c}, and Fatou's lemma, we infer that, for
a subsequence,
\begin{align}
  \|p\|_{L^2(Q_T)} &\le \liminf_{\tau\to 0}\|p^{(\tau)}\|_{L^2(Q_T)} \le C, 
	\nonumber\\
  \|\Pi\mu\|_{L^2(Q_T)} &\le \liminf_{\tau\to 0}\|\Pi\mu^{(\tau)}\|_{L^2(Q_T)}
  \le C,\label{lim.Pimu}
\end{align}
which implies that $|p|$, $|\Pi\mu|<\infty$ a.e.\ in $Q_T$. The fact that 
$p<\infty$ a.e. in $Q_T$ implies that $\sum_{i=1}^n b_i c_i < 1$ a.e. in $Q_T$. 
This property and the relation $|\Pi\mu|<\infty$ a.e. in $Q_T$ imply that
$$ 
  \gamma_i^{(\tau)} := \log c_i^{(\tau)} - \frac{1}{n}\sum_{j=1}^n
	\log c_j^{(\tau)}
$$
is a.e.\ convergent as $\tau\to 0$ for $i=1,\ldots,n$.
Let $(x,t)\in Q_T$ be such that $\gamma_i^{(\tau)}(x,t)$ is convergent for 
$i=1,\ldots,n$ and let 
$$
  J = \Big\{i\in\{1,\ldots,n\}: \lim_{\tau\to 0}c_i^{(\tau)}(x,t) = 0\Big\}.
$$
We want to show that either $J=\emptyset$
or $J = \{1,\ldots,n\}$. Let us assume by contradition that $0<|J|<n$ 
(here $|J|$ is the number of elements in $J$). It follows that
$$ 
  \sum_{i\in J}\gamma_i^{(\tau)}(x,t) 
	= \bigg(1 - \frac{|J|}{n}\bigg)\sum_{i\in J}\log c_i^{(\tau)}(x,t) 
	- \frac{|J|}{n}\sum_{i\notin J}\log c_i^{(\tau)}(x,t) . 
$$
Since $0<|J|<n$, the first sum on the right-hand side diverges to $-\infty$, 
while the second sum is convergent. So the right-hand
side of the above equality is divergent, while the left-hand side is convergent, 
by assumption. This is a contradiction. Thus either the set $J$ is empty or it equals
$\{1,\ldots,n\}$, i.e.\ for a.e. $(x,t)\in Q_T$, either $c_i(x,t)>0$ for 
$i=1,\ldots,n$, or $\ctot(x,t)= 0$. Summarizing up, $c\in\dom\cup\{0\}$.

It follows from \eqref{est.p.mu}--\eqref{est.dc} that 
$\xi\in L^2(Q_T)^n$ exists such that
\begin{align*}
  \na(\Pi\mu^{(\tau)})\rightharpoonup \xi &\quad\mbox{weakly in }L^2(Q_T), \\
  \na p^{(\tau)}\rightharpoonup \na p &\quad\mbox{weakly in }L^2(Q_T), \\
  \tau|\na\mu^{(\tau)}|^3 \to 0 &\quad\mbox{strongly in }L^{4/3}(Q_T), \\
  D(c^{(\tau)})\to D(c) &\quad\mbox{strongly in }L^q(Q_T;\R^{n\times n}),\ q<\infty .
\end{align*}
Moreover, since $c\in\dom\cup\{0\}$, we infer that 
$\xi = \na(\Pi\mu)$ on $\{\ctot>0\}$.
These convergences allow us to perform the limit $\tau\to 0$ in \eqref{3.tau}, obtaining
\begin{equation}
  \pa_t c = \diver(c\na p + D(c)\xi)\qquad\mbox{in }Q_T.\label{eq.xi}
\end{equation}

We will now show that $\ctot>0$ a.e.\ in $Q_{T}$. Then this implies that 
$\xi = \na(\Pi\mu)$ a.e.\ in $Q_T$ and so $D(c)\xi = D(c)\na\mu$, 
since $D(c)\ell = 0$. To this end, summing up the components in \eqref{eq.xi} yields 
(remember that $\sum_{i=1}^{n}D_{ij}= 0$ in $\dom$)
\begin{equation}\label{ctot.tau}
  \pa_t\ctot = \diver\left(\ctot\na p\right) \quad\mbox{in }\Omega.
\end{equation}
Let $\delta>0$. We employ the test function $1/(\delta+\ctot^\Gamma)-1/(\delta+\ctot)$
in \eqref{ctot.tau} giving
$$
  \frac{d}{dt} \int_{\Omega}\left( \frac{\ctot}{\delta+\ctot^\Gamma} 
	+ \log\frac{1}{\delta+\ctot} \right)dx 
  = -\int_{\Omega}\frac{\ctot}{(\delta+\ctot)^2}\na\ctot\cdot\na p dx.
$$
An integration in time in the interval $[0,t]$ (for some $t\in[0,T]$) yields
\begin{align*}
  \int_{\Omega}\log\frac{\delta+\ctot^{0}(x)}{\delta+\ctot(x,t)} dx 
  &+  \int_{\Omega}\frac{\ctot(x,t) - \ctot(x,0)}{\delta+\ctot^\Gamma} dx\\
  &= -\int_0^t\int_{\Omega}\frac{\ctot}{(\delta+\ctot)^2}\na\ctot\cdot\na p dx ds.
\end{align*}
Since the function inside the integral on the right-hand side 
vanishes in the region $\ctot=0$, we can rewrite the above equation as
\begin{align}\label{log.ctot.int}
  \int_{\Omega}\log\frac{\delta+\ctot^{0}(x)}{\delta+\ctot(x,t)} dx 
  & +  \int_{\Omega}\frac{\ctot(x,t) - \ctot(x,0)}{\delta+\ctot^\Gamma} dx\\
  & = -\int_0^t\int_{\{\ctot>0\}}
	\frac{\ctot^{2}}{(\delta+\ctot)^2}\na\log\ctot\cdot\na p dx ds. \nonumber
\end{align}

We want to show that the integral on the right-hand side is bounded from above 
by a constant that depends on $T$ but not on $\delta$.
We show first that $\na\log(p/\ctot)\in L^2 (0,\infty;L^2(\Omega))$.
First, we observe that, because of \eqref{1.kappa},
\begin{align*}
  p &\ge \ctot\bigg(1-\sum_{i,j=1}^n b_i^{-1}a_{ij}b_ic_i\frac{c_j}{\ctot}\bigg)
	= \ctot\bigg(1 - \max_{i,j=1,\ldots,n}(b_i^{-1}a_{ij})\sum_{k=1}^n b_kc_k
	\sum_{\ell=1}^n \frac{c_\ell}{\ctot}\bigg) \\
	&\ge \ctot\left(1-\max_{i,j=1,\ldots,n}(b_i^{-1}a_{ij})\right) = \ctot K.
\end{align*}
This implies that $\ctot\na\log(p/\ctot) = (\ctot/p)\na p - \na\ctot\in 
L^{2} (0,\infty; L^{2}(\Omega))$.
Let $\eta=1/(2\max_{1\leq i\leq n}b_{i})$. We decompose
$$
  |\na\log(p/\ctot)| 
	= |\na\log(p/\ctot)|\chi_{\{\ctot>\eta\}} 
	+ |\na\log(p/\ctot)|\chi_{\{\ctot\leq\eta\}}.
$$
The first term on the right-hand side is bounded in $L^{2} (0,\infty; L^{2}(\Omega))$.
The same holds true for the second term since $1-b\cdot c\geq 1/2$ for $\ctot\leq\eta$  
and $\pa(p/\ctot)/\pa c_i$ is uniformly bounded in $\{\ctot\le \eta\}$. We infer
that $\na\log(p/\ctot)\in L^{2} (0,\infty; L^{2}(\Omega))$, showing the claim.

The right-hand side of \eqref{log.ctot.int} becomes
\begin{align*}
  - & \int_0^t \int_{\{\ctot>0\}}\frac{\ctot^{2}}{(\delta+\ctot)^2}
  \na\log\ctot\cdot\na p\ dx ds \\
  &= -\int_0^t\int_{\{\ctot>0\}}\frac{\ctot^{2}}{(\delta+\ctot)^2}\left( 
	\na\log p - \na\log\frac{p}{\ctot} \right)\cdot\na p\ dx ds \\
  &=-4\int_0^t\int_{\{\ctot>0\}}
	\frac{\ctot^{2}|\na\sqrt{p}|^{2}}{(\delta+\ctot)^2}dxds
  + \int_0^t\int_{\{\ctot>0\}}\frac{\ctot^{2}}{(\delta+\ctot)^2}\na\log\frac{p}{\ctot}
	\cdot\na p\ dx ds \\
	&\leq C.
\end{align*}
Identity \eqref{log.ctot.int}, the bound for $\ctot$, 
and the above estimate imply that 
$$
  \int_\Omega\log\frac{\delta+\ctot(x,0)}{\delta+\ctot(x,t)}dx
  +\int_0^t\int_{\{\ctot>0\}}\frac{4\ctot^{2}}{(\delta+\ctot)^2}
	|\na\sqrt{p}|^{2}dx ds \leq C\quad\mbox{for }\delta>0,\ t>0. 
$$
Taking the limit inferior $\delta\to 0$ on both sides and applying Fatou's lemma,
we obtain
$$
  \int_\Omega\log\frac{\ctot^0}{\ctot(x,t)}dx
  +4\int_0^t\int_{\{\ctot>0\}}|\na\sqrt{p}|^{2}dx ds \leq C,\quad t>0, 
$$
which implies that $\ctot(x,t)>0$ for a.e.\ $x\in\Omega$, $t>0$, and 
$\na\sqrt{p}\in L^{2} (0,\infty; L^{2}(\Omega))$.

As a consequence, $c$ is a weak solution to \eqref{1.c}-\eqref{1.bic}.
Actually, equation \eqref{1.c} is satisfied for test functions in
$L^{4}(0,T;$ $W^{1,4}(\Omega))$ but a density argument shows that the
equation holds in $L^2(0,T;H^1(\Omega))$. 

Next, we show that $\en(c^{(\tau)})\to\en(c)$ strongly in $L^{q}(Q_{T})$
for any $q<2$. Since $c^{(\tau)}\to c$ a.e. in $Q_{T}$ and $c^{(\tau)}$
is uniformly bounded, it suffices to show that the term 
$\ctot^{(\tau)}\log(1-\sum_{i=1}^{n}b_{i}c_{i}^{(\tau)})$ is strongly convergent
(see \eqref{2.H}). This is a consequence of the fact that both
$$
  \ctot^{(\tau)}\bigg|\log\bigg(1-\sum_{i=1}^{n}b_{i}c_{i}^{(\tau)}\bigg)\bigg|
	\leq \frac{\ctot^{(\tau)}}{1-\sum_{i=1}^{n}b_{i}c_{i}^{(\tau)}}
  = p^{(\tau)} + \sum_{i,j=1}^{n}a_{ij}c_{i}^{(\tau)}c_{j}^{(\tau)}
$$ 
and $p^{(\tau)}$ are uniformly bounded in $L^{2}(Q_{T})$. 
The convergence of $(\en(c^{(\tau)})$, together with Fatou's lemma, then
allows us to take the limit $\tau\to 0$ in \eqref{3.ei} and to obtain \eqref{1.ei}.

We point out that, since all the constants $C$ appearing 
in the previous estimates are independent of the
final time $T$, all the bounds that have been found hold true in the time 
interval $(0,\infty)$.

We conclude the existence proof by showing that 
$\log\ctot\in L^{\infty}(0,\infty; L^{2}(\Omega))$.
We use the test function $\Theta_{\delta}(\ctot)-\Theta_{\delta}(\ctot^{\Gamma})$
in \eqref{ctot.tau}, where
$$ 
  \Theta_{\delta}(u) := \frac{1}{u+\delta}\log\left( \frac{u+\delta}{M + \delta} 
	\right),\quad M = \frac{1}{\min_{1\leq i\leq n}b_{i}}.
$$
Notice that $\ctot\leq M$ a.e.\ in $\Omega$, $t>0$. It follows that
\begin{align*}
  \frac{1}{2} & \int_{\Omega}\left|\log\left( \frac{\ctot(x,t)+\delta}{M + \delta} 
	\right)\right|^{2}dx 
  - \frac{1}{2}\int_{\Omega}\left|\log\left( \frac{\ctot(x,0)+\delta}{M + \delta} 
	\right)\right|^{2}dx\\
  &\phantom{xx}{} - \Theta_\delta(\ctot^\Gamma)
	\int_\Omega(\ctot(x,t)-\ctot(x,0))dx\\
  &= -\int_{0}^{t}\int_{\Omega}\frac{\ctot}{\ctot+\delta}
	\left( 1 - \log\left( \frac{\ctot+\delta}{M + \delta} \right) \right)
  \na p\cdot\na\log\ctot dx ds.
\end{align*}
Inserting $\na\log\ctot = \na\log p - \na\log(p/\ctot)$ on the right-hand side,
the first term is nonpositive (because of $\ctot\le M$, we have
$1-\log((\ctot+\delta)/(M+\delta))\ge 0$) and we end up with
\begin{align*}
  \frac{1}{2} & \int_{\Omega}\left|\log\left( \frac{\ctot(x,t)+\delta}{M + \delta} 
	\right)\right|^{2}dx 
  - \frac{1}{2}\int_{\Omega}\left|\log\left( \frac{\ctot^{0}(x)+\delta}{M + \delta} 
	\right)\right|^{2}dx\\
  &\leq C + \int_{0}^{t}\int_{\Omega}\frac{\ctot}{\ctot+\delta}
	\left( 1 - \log\left( \frac{\ctot+\delta}{M + \delta} \right) \right)
  \na p\cdot\na\log\frac{p}{\ctot} dx ds = C + I_{1} + I_{2},
\end{align*}
where the constant $C>0$ estimates the term proportional to 
$\Theta_\delta(\ctot^\Gamma)$ and
\begin{align*}
  I_{1} &:= \int_0^t\int_{\{\ctot\le\eta\}}
	\frac{2\ctot\sqrt{p}}{\ctot+\delta}
	\left( 1 - \log\left( \frac{\ctot+\delta}{M + \delta} \right) \right)
  \na\sqrt{p}\cdot\na\log\frac{p}{\ctot} dx ds,\\
  I_{2} &:= \int_0^t\int_{\{\ctot>\eta\}}\frac{\ctot}{\ctot+\delta}
	\left( 1 - \log\left( \frac{\ctot+\delta}{M + \delta} \right) \right)
  \na p\cdot\na\log\frac{p}{\ctot} dx ds,
\end{align*}
and $\eta=1/(2\min_{1\leq i\leq n}b_{i})$.

It is straightforward to see that $\sqrt{p}\log((\ctot+\delta)/(M + \delta))$ 
is uniformly bounded with respect to $\delta$ in the region $\{\ctot\le\eta\}$.
Since $\na\sqrt p \in L^{2}(0,\infty; L^{2}(\Omega))$, we deduce that $I_{1}$ 
is uniformly bounded with respect to $\delta$. Furthermore, the regularity
$\na p \in L^{2}(0,\infty; L^{2}(\Omega))$ implies that $I_{2}$ is uniformly 
bounded with respect to $\delta$. As a consequence, 
$$
  \int_{\Omega}\left|\log\left( \frac{\ctot+\delta}{M + \delta} \right)\right|^{2}dx 
	\leq C,\quad t>0.
$$
Taking the limit inferior $\delta\to 0$ on both sides of the above inequality 
and applying Fatou's Lemma, we conclude that 
$\log\ctot\in L^{\infty}(0,\infty; L^{2}(\Omega))$.
This finishes the proof of part (i).


%


\subsection{Large-time asymptotics}

We first show that, for some generic constant $C>0$,
\begin{equation}
  |(\ctot\en'')^{-1}\mu| \leq C\left(1 + p + |\log\ctot| \right).\label{est.lt}
\end{equation}
Let $w := (\ctot\en'')^{-1}\mu$, i.e. $\ctot\en'' w = \mu$. It follows from 
Lemma \ref{lem.Hess} that
\begin{align*}
  \ctot\sum_{i=1}^{n}\frac{w_{i}^{2}}{c_{i}}
	&\leq \frac{1}{\kappa} w\cdot(\ctot\en'')w = \frac{1}{\kappa}\mu\cdot w 
  = \frac{1}{\kappa}\sum_{i=1}^{n}\frac{\sqrt{c_{i}}}{\sqrt{\ctot}}\mu_{i} 
	\frac{\sqrt\ctot}{\sqrt{c_{i}}}w_{i}\\
  &\leq \frac{1}{\kappa}\bigg( \sum_{j=1}^{n}
	\frac{\sqrt{c_{j}}}{\sqrt{\ctot}}|\mu_{j}|\bigg) 
	\bigg(\ctot \sum_{i=1}^{n}\frac{w_{i}^{2}}{c_{i}} \bigg)^{1/2}.
\end{align*}
This gives
\begin{equation}
  |w|\leq \bigg( \ctot\sum_{i=1}^{n}\frac{w_{i}^{2}}{c_{i}} \bigg)^{1/2}
	\leq \frac{1}{\kappa}\sum_{j=1}^{n}
	\frac{\sqrt{c_{j}}}{\sqrt\ctot}|\mu_{j}|. \label{est.w}
\end{equation}
It remains to estimate the right-hand side. 
We claim that $0\le -\log(1-b\cdot c)\le C(1+p)$. Indeed, with
$\eta=1/(2\max_{i=1,\ldots,n}b_i)$, we have
$$
  -\log(1-b\cdot c)
	\le \chi_{\{\ctot\le\eta\}}\log\frac{1}{1-b\cdot c}
	+ \frac{\ctot}{\eta}\chi_{\{\ctot>\eta\}}\log\frac{1}{1-b\cdot c}.
$$
The first term on the right-hand side is bounded since $\ctot\le\eta$ implies
that $1-b\cdot c\ge 1/2$. Then, since $\log(1/z)\le 1/z$ for $z>0$,
$$
  0\le -\log(1-b\cdot c)\le C + 2\max_{i=1,\ldots,n}b_i\frac{\ctot}{1-b\cdot c}
	\le C(1+p).
$$
Hence, by definition \eqref{1.mu} of $\mu_i$, 
\begin{align*}
  \frac{\sqrt{c_{i}}}{\sqrt\ctot}|\mu_{i}| 
	&\leq C(1+p) + \frac{\sqrt{c_{i}}}{\sqrt\ctot}|\log c_{i}| \\
  &\leq C(1+p) + 2 \frac{\sqrt{c_{i}}}{\sqrt\ctot}\left|
	\log\frac{\sqrt{c_{i}}}{\sqrt\ctot}\right| 
	+ \frac{\sqrt{c_{i}}}{\sqrt\ctot}|\log\ctot|,
\end{align*}
and therefore,
\begin{align}
& \frac{\sqrt{c_{i}}}{\sqrt\ctot}|\mu_{i}| \leq C(1+ p + |\log\ctot|).
\label{est.lt.mu}
\end{align}
Putting together \eqref{est.w} and \eqref{est.lt.mu} yields \eqref{est.lt}.

A computation shows that $\en(c)=-p(c)+\sum_{i=1}^n c_i\mu_i$
(in fact, this is the Gibbs-Duhem relation, see \eqref{m.gde}) and
$\na\en(c)=c\cdot\na\mu$ (this follows from \eqref{2.nap}).
Since $c^\Gamma=\Phi^{-1}(\mu)|_{\mu=0}$, we have $\en(c^\Gamma)=-p(c^\Gamma)$.
We use the fact that $c_i$ varies in a bounded domain and employ
the Poincar\'e inequality with constant $C_P$
and the identity $\na\mu=\en''(c)\na c$ to find that
\begin{align*}
  \int_\Omega \en^*(c)dx 
	&\leq C_{P}\int_\Omega |\na\en(c)| dx 
  = C_{P}\int_{\Omega} |\mu\cdot\na c| dx \\
	&= C_{P}\int_{\Omega}\sum_{i,j=1}^n\big|\ctot((\en'')^{-1})_{ij}\mu_i\ctot \na\mu_j
  \big|dx	\\
  &\leq C_P\|(\ctot\en'')^{-1}\mu\|_{L^{2}(\Omega)}\|\ctot\na\mu\|_{L^{2}(\Omega)},
\end{align*}
which, thanks to \eqref{est.lt}, leads to
$$
  \bigg( \int_\Omega \en^*(c)dx \bigg)^{2}
	\leq C\big(1 + \|p\|_{L^{2}(\Omega)}^{2} + \|\log\ctot\|_{L^{2}(\Omega)}^{2} \big)
	\|\ctot\na\mu\|_{L^{2}(\Omega)}^{2}.
$$
Taking into account \eqref{1.ei} and Lemma \ref{lem.B}, we obtain
$$ 
  \|\ctot\na\mu\|_{L^{2}(\Omega)}^{2} 
	\le C\int_\Omega\na\mu:B(c)\na\mu dx
	= C\bigg( -\frac{d}{dt}\int_\Omega \en^*(c)dx \bigg). 
$$
We deduce from the above inequalities and the facts that 
$p\in L^{2}(0,T; L^{2}(\Omega))$ and $\log\ctot\in L^{\infty}(0,T; L^{2}(\Omega))$,
$$
  \bigg( \int_\Omega \en^*(c)dx \bigg)^{2}
	\leq C(\Psi(t) + 1)\bigg( -\frac{d}{dt}\int_\Omega \en^*(c)dx \bigg),
$$
where $\Psi = \|p\|_{L^{2}(\Omega)}^{2}\in L^{1}(0,\infty)$. 
A nonlinear Gronwall inequality shows that
\begin{equation}\label{lt.S}
  \int_\Omega \en^*(c)dx \leq \frac{S_{0}}{1 + C S_{0}\psi(t)},\quad t>0,
\end{equation}
where $S_{0} := \int_\Omega \en^*(c^{0})dx$ and 
$\psi(t) := \int_{0}^{t}(1 + \Psi(\tau))^{-1}d\tau$.

We define now $f : (0,1]\to\R$, $f(x) = 1/x - 1$. Clearly, $f$ is decreasing 
and convex. Jensen's inequality and the fact that $\Phi\in L^1(0,\infty)$ yield
$$
  f\bigg(\frac{\psi(t)}{t}\bigg)
	\leq \frac{1}{t}\int_{0}^{t}f\left( (1 + \Psi(\tau))^{-1} \right)d\tau 
	= \frac{1}{t}\int_{0}^{t}\Psi(\tau)d\tau
	\leq \frac{C}{t}.
$$
Since $f$ (and also its inverse $f^{-1}$) is decreasing, it follows that
$$ 
  \frac{\psi(t)}{t}
	\geq f^{-1}\left(\frac{C}{t} \right)
	= \frac{1}{1 + C t^{-1}} 
	\geq \frac{1}{2} \qquad\mbox{for }t\ge C. 
$$
We conclude from this fact and \eqref{lt.S} that
\begin{equation}\label{lt.S.2}
  \int_\Omega \en^*(c)dx \leq \frac{C}{1+t},\quad t>0.
\end{equation}
By Lemma \ref{lem.Hess}, the Hessian
$\en''$ is positive definite. Moreover, $\en'(c^\Gamma)=\mu|_{c=c^\Gamma}=0$.
Thus, a Taylor expansion shows that
$$
  \int_\Omega F^*(c)dx = \int_\Omega\bigg(\en'(c^\Gamma)\cdot(c-c^\Gamma)
	+ \frac12(c-c^\Gamma):\en''(\xi)(c-c^\Gamma)\bigg)dx
	\ge \frac{\kappa}{2}\int_\Omega|c-c^\Gamma|^2 dx,
$$
where $\kappa>0$ is specified in \eqref{1.kappa}. This finishes the proof of 
Theorem \ref{thm.ex1}.


\subsection{Proof of Corollary \ref{coro.ex}}

The existence proof is similar to that one of Theorem \ref{thm.ex1}. 
The main difference is that we lose the information
on the chemical potentials $\mu_{1},\ldots,\mu_{n}$ due to the possible 
degeneracy of $D$ (since $\en''$ is unbounded). 
However, thanks to the simple structure of \eqref{1.c.simple},
we do not need uniform estimates on $\mu_{1},\ldots,\mu_{n}$ in order to be able
to pass to the deregularization limit. 

Compared to \eqref{3.disc}, we employ a slightly different time 
discretization to overcome the difficulty that $D$ is not strictly positive definite:
\begin{equation}\label{coro.3.disc}
  \frac{c_i^{k}-c_i^{k-1}}{\tau}
	= \diver\bigg(\sum_{j=1}^n B_{ij}^{k}\na\mu_j^{k}\bigg)
	+ \tau\diver\big( |\na\mu_i^k|^{2} \na\mu_i^k \big) \quad\mbox{in }\Omega,\ 
	i=1,\ldots,n.
\end{equation}
The uniform estimates for $p^{k}$, $c^{k}$ provided by \eqref{est.p.mu}, \eqref{est.c},
repsectively, still hold. Lemma \ref{lem.Hess} allows us to infer 
that $\na\mu^{k}\cdot\na c^{k} = \na c^{k}\cdot(\en^{k})''\na c^{k} 
\geq \kappa |\na\sqrt{c^{k}}|^{2}$. The limit mass densities $c_{1},\ldots,c_{n}$
satisfy $c\in\overline{\dom}$. The proof that $c\in\dom$ is slightly different 
than in the proof of Theorem \ref{thm.ex1}. Indeed, 
the $L^{\infty}(0,\infty; L^{1}(\Omega))$ bound for $\en$ implies that 
$\sum_{i=1}^{n}b_{i}c_{i}<1$ a.e. in $\Omega$, $t>0$. 
This fact and the previous bounds allow us
to take the limit $\tau\to 0$ in \eqref{coro.3.disc} and to obtain 
\eqref{1.c.simple} together with the properties
$$
  c_i-c_i^\Gamma \in L^2(0,\infty ;H^1(\Omega))\cap 
	H^1(0,\infty ;H^{1}(\Omega)'), \quad\na\sqrt{c_i},\ \na p 
	\in L^{2}(0,\infty; L^{2}(\Omega)).
$$

In order to prove that $c_{i}>0$ a.e.\ in $\Omega$, $t>0$, for $i=1,\ldots,n$,
we choose $\delta>0$, employ the test function
$(\delta+c_{i}^\Gamma)^{-1}-(\delta+c_{i})^{-1}$ in \eqref{1.c.simple}, 
and sum over $i=1,\ldots,n$:
\begin{align*}
  \frac{d}{dt}\int_{\Omega} & \sum_{i=1}^{n}
	\left(\frac{c_i}{\delta+c_i^\Gamma} - \log(\delta + c_{i}) \right) dx\\
  &= -\sum_{i}\int_{\Omega}\left(\frac{(1+\eps\beta)c_{i}}{c_{i}+\delta}
	\na p\cdot\na\log(c_{i}+\delta) + \alpha|\na\log(c_{i}+\delta)|^{2}\right)dx. 
\end{align*}
Since $\alpha$ is strictly positive, $\beta$ is bounded, and 
$\na p \in L^{2}(0,\infty; L^{2}(\Omega))$, by applying Young's inequality 
and integrating in time, we conclude that
$$
  \|\log(\delta + c_{i})\|_{L^{\infty}(0,\infty; L^{1}(\Omega))} 
	+ \|\na\log(\delta+c_{i})\|_{L^{2}(0,\infty; L^{2}(\Omega))}
	\leq C,\quad i=1,\ldots, n. 
$$
Fatou's Lemma allows us to conclude that 
$\log c_{i}\in L^{\infty}(0,\infty; L^{1}(\Omega))$ and
$\na\log(\delta+c_{i})\in L^{2}(0,\infty; L^{2}(\Omega))$
for $i=1,\ldots, n$; in particular $c_{i}>0$ a.e.\ in $\Omega$, $t>0$. 
The free energy inequality \eqref{1.ei} follows with the same argument
as in the proof of Theorem \ref{thm.ex1}. This finishes the proof of Corollary
\ref{coro.ex}.


\section{Proof of Theorem \ref{coro}}\label{sec.coro}


\subsection{Integral inequality} 

Let $\delta>0$ be arbitrary, and let 
$z=(z_{1},\ldots,z_{n-1})$, $z_{i}=c_{i}/\ctot$, 
$z^{\delta}=(z_{1}^{\delta},\ldots,z_{n-1}^{\delta})$, 
$z_{i}^{\delta}=(\delta+c_{i})/(\delta+\ctot)$ for $i=1,\ldots,n-1$.
Moreover, let $\psi_{\delta}(c) = (\ctot + {\delta})f(z^{\delta})$ for $c\in\dom$, 
where $f : [0,1]^{n-1}\to\R$ satisfies the assumptions of Theorem \ref{coro}.
A simple computation yields
\begin{align*}
  \frac{\pa z_k^\delta}{\pa c_i} 
	&= \frac{\delta_{ik}}{\ctot+\delta} - \frac{c_k+\delta}{(\ctot+\delta)^2},\\
  \frac{\pa^{2}\psi_{\delta}(c)}{\pa c_{i}\pa c_{j}} 
  &= (\ctot+\delta)\sum_{k,s=1}^{n-1}\frac{\pa^{2} f}{\pa z_{k}\pa z_{s}}
	\frac{\pa z_{k}^{\delta}}{\pa c_{i}}\frac{\pa z_{s}^{\delta}}{\pa c_{j}},
	\quad i,j=1,\ldots,n.
\end{align*}
Employing $\pa\psi_{\delta}(c)/\pa c_{i} - \pa\psi_{\delta}(c^{\Gamma})/\pa c_{i}$ 
as a test function in \eqref{1.c.simple} leads to
\begin{equation}\label{dpsi}
  \int_{\Omega}\left( \psi_{\delta}(c(x,t)) - \psi_{\delta}(c(x,0))\right) dx 
  -\sum_{i=1}^{n}\frac{\pa\psi_{\delta}}{\pa c_{i}}(c^{\Gamma})
	\int_{\Omega}(c_{i}(x,t)-c_{i}(x,0)) dx = -J_1-J_2,
\end{equation}
where
\begin{align*}
  J_{1} &= \sum_{i,j=1}^{n}\int_{0}^{t}\int_{\Omega}(1+\eps\beta(c))\na c_{j}
	\cdot\na p (\ctot+\delta)\sum_{k,s=1}^{n-1}\frac{\pa^{2} f}{\pa z_{k}\pa z_{s}}
	\frac{\pa z_{k}^{\delta}}{\pa c_{i}}\frac{\pa z_{s}^{\delta}}{\pa c_{j}}
  c_{i} dx ds,\\
  J_{2} &= \eps\sum_{i,j=1}^{n}\int_{0}^{t}\int_{\Omega}\alpha(c)\na c_{j}
	\cdot\na c_i(\ctot+\delta)\sum_{k,s=1}^{n-1}\frac{\pa^{2} f}{\pa z_{k}\pa z_{s}}
	\frac{\pa z_{k}^{\delta}}{\pa c_{i}}\frac{\pa z_{s}^{\delta}}{\pa c_{j}} dx ds.
\end{align*}
It holds that
$$ 
  J_{2} = \eps\int_{0}^{t}\int_{\Omega}\alpha(c)(\ctot+\delta)
  \sum_{k,s=1}^{n-1}\frac{\pa^{2} f}{\pa z_{k}\pa z_{s}}\na z_{k}^{\delta}
	\cdot \na z_{s}^{\delta}\ dx ds, 
$$
and so $J_{2}\geq 0$, since $f$ is convex. 
We show now that $|J_{1}|\to 0$ as $\delta\to 0$. We compute
\begin{align*}
  \sum_{i=1}^n\frac{\pa z_k^\delta}{\pa c_i}c_i 
	= \frac{\delta(c_k-\ctot)}{(\ctot+\delta)^2}\to 0
	& \quad\mbox{a.e. in }\Omega\times(0,\infty)\mbox{ as }\delta\to 0,\\
  (\ctot+\delta)\left|\frac{\pa z_k^\delta}{\pa c_j}\right| 
	+ \left|\sum_{i=1}^n\frac{\pa z_k^\delta}{\pa c_i}c_i\right|
	\leq C &\quad\mbox{a.e. in }\Omega\times(0,\infty),~~ j=1,\ldots,n.
\end{align*}
The above relations, together with the boundedness of $\beta$ and $f''$, 
allow us to apply the dominated convergence theorem and deduce that
$|J_{1}|\to 0$ as $\delta\to 0$. Moreover, \eqref{1.f} implies that 
$\pa\psi_{\delta}(c^\Gamma)/\pa c_{i}\to 0$ as $\delta\to 0$,
$i=1,\ldots,n$. The continuity and boundedness of $f$ imply that 
$\psi_{\delta}(c(\cdot,0))\to \ctot^0 f(c_1^0/\ctot^0,\ldots,c_{n-1}^0,\ctot^0)$
in $L^1(\Omega)$ as $\delta\to 0$. Taking the limit inferior $\delta\to 0$ 
on both sides of \eqref{dpsi} and exploiting all the convergence relations
as well as the nonnegativity of $J_2$ yield
$$ 
  \liminf_{\delta\to 0}\int_\Omega\psi_\delta(c(x,t)) dx 
  \leq \int_\Omega \ctot^0 f\bigg(\frac{c_1^0}{\ctot^0},\ldots,
	\frac{c_{n-1}^0}{\ctot^0}\bigg)dx. 
$$
Finally, by Fatou's Lemma, we conclude that \eqref{1.ii2} holds.


\subsection{Maximum principle} 

The final statement of Theorem \ref{coro}
is a consequence of the following lemma.

\begin{lemma}\label{lem.max}
Let $c_i$, $c_i^0\in L^\infty(\Omega)$ for $i=1,\ldots,n$ 
be positive functions such that
$c_i=c_i^0=c_i^\Gamma$ on $\pa\Omega$ for some constant $c_i^\Gamma>0$, $i=1,\ldots,n$. 
Let a constant $m\in(0,c_1^\Gamma/\ctot^\Gamma)$ exist such that 
$c_1^0/\ctot^0\ge m$
in $\Omega$. Finally, assume that \eqref{1.ii2} holds for any $f\in C^2(0,1)$ 
satisfying \eqref{1.f}. Then $c_1/\ctot\ge m$ in $\Omega$.
\end{lemma}

\begin{proof}
Let $f(x) = (m-x)_{+}^{3}$ for $0\leq x\leq 1$. 
Clearly $f\in C^2(0,1)$ satisfies \eqref{1.f}.
Taking into account the assumptions of the lemma, 
we deduce that \eqref{1.ii2} holds for
the above choice of $f$. Since $c_1^0/\ctot^0\ge m$, the right-hand side
vanishes. Because of the nonnegativity of $f$ and the positivity of 
$c_i$, we infer that $0=f(c_1/\ctot)=(m-c_1/\ctot)_+^3$ in $\Omega$ and
$c_1/\ctot\ge m$ in $\Omega$, concluding the proof.
\end{proof}


\section{Proof of Theorem \ref{thm3}}\label{sec.thm3}

\subsection{Derivation of the evolution equation for $p$}

We multiply \eqref{1.c} by $\pa p/\pa c_i$, sum over $i=1,\ldots,n$, and
compute in the sense of distributions:
\begin{equation}\label{7.p}
  \pa_t p = \sum_{i=1}^n \frac{\pa p}{\pa c_i}\diver(c_i\na p)
	= \sum_{i=1}^n \frac{\pa p}{\pa c_i}(\na c_i\cdot\na p+c_i\Delta p)
	= |\na p|^2 + \widetilde{D}\Delta p,
\end{equation}
where $\widetilde{D}=\sum_{i=1}^n (\pa p/\pa c_i)c_i$. 
Because of the Gibbs-Duhem relation
\eqref{m.gde}, it follows that
$\pa p/\pa c_i=\sum_{j=1}^n c_j\pa^2\en/\pa c_i\pa c_j$, and consequently,
$$
  \widetilde{D} = \sum_{i,j=1}^n c_ic_j\frac{\pa^2\en}{\pa c_i\pa c_j}.
$$

We claim that $\widetilde{D}\ge p$. Indeed, definition \eqref{2.H} leads to
$$
  \frac{\pa^2\en}{\pa c_i\pa c_j}
	= (b_i+b_j)\sigma + b_ib_j\ctot\sigma^2 + \frac{\delta_{ij}}{c_i} - a_{ij},
	\quad \sigma = \frac{1}{1-\sum_{i=1}^n b_ic_i}\ge 1.
$$
Then
\begin{align*}
  \widetilde{D} &= 2\ctot\sigma\sum_{i=1}^n b_ic_i
	+ \ctot\sigma^2 \bigg(\sum_{i=1}^n b_ic_i\bigg)^2 + \ctot
	- \sum_{i,j=1}^n a_{ij}c_ic_j \\
	&= \ctot\bigg(1 + \sigma\sum_{i=1}^n b_ic_i\bigg)^2 - \sum_{i,j=1}^n a_{ij}c_ic_j
	= \frac{\ctot}{(1-\sum_{i=1}^n b_ic_i)^2} - \sum_{i,j=1}^n a_{ij}c_ic_j \ge p.
\end{align*}

\subsection{Lower bound for the pressure}

We show that $p\ge m$ in $\Omega$, $t>0$, where 
$m=\min\{\min_\Omega p(c^0),p^\Gamma\}>0$. Then equation \eqref{7.p} is
uniformly parabolic. Using $(p-m)_-=\min\{0,p-m\}$ as a test function in 
\eqref{7.p} and integrating by parts gives
\begin{align*}
  \frac12 \frac{d}{dt}\int_\Omega(p-m)_-^2 dx
	&= -\int_\Omega\big(\widetilde{D}-(p-m)_-\big)|\na(p-m)_|^2 dx \\
	&\phantom{xx}{}- \int_\Omega(p-m)_-\na \widetilde{D}\cdot\na(p-m)_- dx.
\end{align*}
Since $\widetilde{D}\ge p$, it follows that 
$\widetilde{D}-(p-m)_-\ge D-(p-m)\ge m$. Thus, together
with Young's inequality, we find that
\begin{equation}\label{mah}
  \frac12 \frac{d}{dt}\int_\Omega(p-m)_-^2 dx
	\le -\frac{m}{2}\int_\Omega|\na(p-m)_-|^2
	+ \frac{1}{2m}\int_\Omega |\na \widetilde{D}|^2(p-m)_-^2 dx.
\end{equation}
The second term on the right-hand side can be bounded by means of the Cauchy-Schwarz,
Gagliardo-Nirenberg (with constant $C_{GN}>0$, using $d=2$), and Young inequalities:
\begin{align*}
  \int_\Omega & |\na \widetilde{D}|^2 (p-m)_-^2 dx 
	\leq \|\na \widetilde{D}\|_{L^4(\Omega)}^2\|(p-m)_-\|_{L^4(\Omega)}^2\\
  &\leq C_{GN} \|\na \widetilde{D}\|_{L^4(\Omega)}^2 
	\|(p-m)_-\|_{L^2(\Omega)}\|(p-m)_-\|_{H^1(\Omega)}\\
  &\leq \frac{m^2}{2}\int_\Omega |\na(p-m)_-|^2 dx 
  + \left( \frac{C_{GN}^2}{2 m^2} \|\na \widetilde{D}\|_{L^4(\Omega)}^4 
	+ \frac{m^2}{2}\right)\int_\Omega(p-m)_-^2 dx.
\end{align*}
So \eqref{mah} implies that
\begin{align*}
  \frac{d}{dt}\int_\Omega(p-m)_-^2 dx
  \le -\frac{m}{2}\int_\Omega|\na(p-m)_-|^2
  + \left( \frac{C_{GN}^2}{2 m^3} \|\na \widetilde{D}\|_{L^4(\Omega)}^4 
	+ \frac{m}{2}\right)\int_\Omega (p-m)_-^2 dx.
\end{align*}
In view of our regularity assumptions on $\na c_i$, we have
$\na\widetilde{D}\|_{L^4(\Omega)}^4\in L^1_{\rm loc}(0,\infty)$,
and we conclude with Gronwall's lemma that $(p-m)_-=0$,
i.e.\ $p\ge m$ in $\Omega$, $t>0$.

\subsection{Gradient estimate for the pressure}

We multiply \eqref{7.p} with $\Delta p$ and use the lower bound 
$\widetilde{D}\ge p\ge m$
and the Gagliardo-Nirenberg inequality with $d=2$:
\begin{align}
  \frac12\frac{d}{dt} & \int_\Omega|\na p|^2 dx + m\|\Delta p\|_{L^2(\Omega)}^2
	\le \int_\Omega |\na p|^2\Delta p dx 
	\le \|\na p\|_{L^4(\Omega)}^2\|\Delta p\|_{L^2(\Omega)} \nonumber \\
	&\le C_{GN}^2\|\na p\|_{H^1(\Omega)}^{1/2}\|\na p\|_{L^2(\Omega)}^{1/2}
	\|\Delta p\|_{L^2(\Omega)} \nonumber \\
	&= C_{GN}^2\big(\|\na^2 p\|_{L^2(\Omega)}^2 + \|\na p\|_{L^2(\Omega)}^2\big)^{1/2}
	\|\na p\|_{L^2(\Omega)}\|\Delta p\|_{L^2(\Omega)}. \label{7.aux}
\end{align}
We claim that $\|\na p\|_{L^2(\Omega)}\le C_0\|\Delta p\|_{L^2(\Omega)}$
for some constant $C_0>0$ which only depends on $\Omega$ and $d$. 
Because of $p=p^\Gamma=\mbox{const.}$ on $\pa\Omega$, we have
$\int_\Omega \na pdx=\int_{\pa\Omega}p\nu ds = p^\Gamma\int_{\pa\Omega}\nu dx=0$,
which implies that
$$
  \|\na p\|_{L^2(\Omega)} = \bigg\|\na p - \frac{1}{\mbox{meas}(\Omega)}
	\int_\Omega \na p dx\bigg\|_{L^2(\Omega)}
	\le C_P\|\na^2 p\|_{L^2(\Omega)},
$$
where $C_P>0$ is the Poincar\'e constant.
The function $v:=p-p^\Gamma$ satisfies $\Delta v=f:=\Delta p$ in $\Omega$
and $v=0$ on $\pa\Omega$. By elliptic regularity,
\begin{equation}\label{7.ell}
  \|\na^2 p\|_{L^2(\Omega)} \le \|v\|_{H^2(\Omega)} \le C_E\|f\|_{L^2(\Omega)}
	= C_E\|\Delta p\|_{L^2(\Omega)}
\end{equation}
for some constant $C_E>0$, and therefore,
\begin{equation}\label{7.poin}
  \|\na p\|_{L^2(\Omega)} \le C_PC_E\|\Delta p\|_{L^2(\Omega)}.
\end{equation}

We infer from \eqref{7.ell} and \eqref{7.poin} that \eqref{7.aux} becomes
$$
  \frac12\frac{d}{dt} \int_\Omega|\na p|^2 dx + m\|\Delta p\|_{L^2(\Omega)}^2
	\le C_{GN}^2C_E(1+C_P^2)^{1/2}\|\Delta p\|_{L^2(\Omega)}
	\|\na p\|_{L^2(\Omega)}\|\Delta p\|_{L^2(\Omega)},
$$
and hence,
\begin{equation}\label{7.aux2}
  \frac{d}{dt}\int_\Omega|\na p|^2 dx 
	+ 2\big(m-C_1\|\na p\|_{L^2(\Omega)}\big)
	\|\Delta p\|_{L^2(\Omega)}^2 \le 0,
\end{equation}
where $C_1=C_{GN}^2 C_E(1+C_P)^{1/2}$. Let $0<K_0<1/C_1$. Then,
by assumption, $\lambda:=m-C_1\|\na p(c^0)\|_{L^2(\Omega)}>0$.
Since $|\na p|\in C^0([0,\infty);L^2(\Omega))$, the coefficient remains positive
in a small time interval $[0,t^*)$. As a consequence, 
$t\mapsto\|\na p(c(t))\|_{L^2(\Omega)}^2$ is nonincreasing in $[0,t^*)$.
A standard prolongation argument then implies that 
$m-C_1\|\na p(c(t))\|_{L^2(\Omega)}>0$ and
$t\mapsto \|\na p(c(t))\|_{L^2(\Omega)}^2$ is nonincreasing for all $t>0$.
In particular,
$$
  m-C_1\|\na p(c(t))\|_{L^2(\Omega)} \ge \lambda.
$$
{}From this fact and estimates \eqref{7.aux2} and \eqref{7.poin}, we deduce that
$$
  \frac{d}{dt}\int_\Omega|\na p|^2 dx 
	\le -2\lambda\|\Delta p\|_{L^2(\Omega)}^2
	\le -2\lambda (C_PC_E)^{-2}\|\na p\|_{L^2(\Omega)}^2,
$$
and Gronwall's lemma allows us to conclude.


\section{Numerical experiments}\label{sec.num}

We solve system \eqref{1.c}-\eqref{1.p} numerically in one space dimension
for the case $\eps=0$ and $n=2$, imposing Dirichlet and homogeneous Neumann
boundary conditions for $p$. 
Let $\{t_k:k\ge 0\}$ with $t_0=0$ be a discretization of the time 
interval $[0,\infty)$ and $\{x_j:0\le j\le N\}$ with $N\in\N$, $x_j=jh$, and $h=1/N$, 
be a uniform discretization of the space interval $\Omega=(0,1)$. 
We set $\tau_k=t_k-t_{k-1}$ for $k\ge 1$. 
For the discretization of \eqref{1.c}, we
distinguish between the two boundary conditions.

\subsection{Homogeneous Neumann boundary conditions}

We employ the staggered grid $y_j=x_{j-1/2}=(x_j+x_{j-1})/2$ and
denote by $c_{i,j}^k$ and $p_j^k$ the approximations of $c_i(y_j,t_k)$ and
$p(y_j,t_k)$, respectively. The values at the interior points are the unknowns of 
the problem, while the values at the boundary points are determined according to
$$
  c_{i,0}^k = c_{i,1}^k, \quad c_{i,N+1}^k = c_{i,N}^k, \quad
	k\ge 0,\ i=1,2.
$$
The initial condition is discretized by 
$$
  c_{i,j}^0 = \frac12\big(c_i^0(x_j)+c_i^0(x_{j-1})\big), \quad
	j=1,\ldots,N,\ i=1,2.
$$
Approximating the time derivative by the implicit Euler scheme and the diffusion
flux $J_i := -c_i\pa_x p$ at $(y_{j+1/2}, t_k)$ by the implicit upwind scheme
\begin{equation}\label{Upwind}
  J_{i,j+1/2}^{k} = c_{i,j}^k\max\{v_{j+1/2}^k, 0\} 
	+ c_{i,j+1}^k\min\{v_{j+1/2}^k, 0\},\quad v_{j+1/2}^k = -\frac{p_{j+1}-p_j}{h},
\end{equation}
the finite-difference scheme for \eqref{1.c} becomes
\begin{equation}
  \frac{1}{\tau_k}\big(c_{i,j}^{k}-c_{i,j}^{k-1}\big) 
	+ \frac{1}{h}\big(J_{i,j+1/2}^{k}-J_{i,j-1/2}^{k}\big) = 0, \label{6.neumann}
\end{equation}
where $1\le j\le N$, $k\ge 1$, $i=1,2$. To be consistent with the
boundary conditions, we define $p_0^k=p_1^k$ and $p_{N+1}^k=p_N^k$, $k\ge 0$.

\subsection{Dirichlet boundary conditions}

Here, we do not need to employ the staggered grid, so we use the original 
grid $\{x_j :  0\leq j\leq N\}$.
The implicit scheme \eqref{Upwind}-\eqref{6.neumann} works also in this situation, 
with the only difference that the boundary conditions are simply given by
$c_{i,0}^k=c_i(0,t_k)$, $c_{i,N}^k=c_i(1,t_k)$, and the initial condition
is defined by $c_{i,j}^0=c_i^0(x_j)$.

\subsection{Iteration procedure}

The nonlinear equations are solved by using the {\sc Matlab}
function {\tt fsolve}, with $c_{i,j}^{k-1}$ as the initial guess.
The time step $\tau_k$ is chosen in an adaptive way. At each time iteration,
once the new iterate $c_{i,j}^k$ is computed, the relative difference
between two consecutive iterates,
$$
  \rho^k = \sqrt{\frac{\sum_{i=1}^2\sum_{j=1}^N
	|c_{i,j}^k-c_{i,j}^{k-1}|^2}{\sum_{i=1}^2\sum_{j=1}^N|c_{i,j}^{k-1}|^2}},
$$
is evaluated and compared to the maximal tolerance $\rm{tol}_M$. If 
$\rho^k\ge \rm{tol}_M$, the iterate is rejected, the time step $\tau_k$ is
halved, and the step is repeated. Otherwise, the iterate is accepted.
Before the next iterate is computed, $\rho^k$ is compared to the minimal
tolerance $\rm{tol}_m$ (with $\rm{tol}_m<\rm{tol}_M$). If $\rho^k<\rm{tol}_m$,
the time step is increased by a factor $5/4$. Otherwise, $\tau_k$ is kept
unchanged. In the simulations, we have chosen the values
$\rm{tol}_m=4\cdot 10^{-4}$, $\rm{tol}_M=6\cdot 10^{-4}$, and $N=201$.

\subsection{Numerical results}

We present the results of four numerical simulations, referring to the different
boundary conditions and different choices of the parameters, namely
$$
  b_1 = 1, \quad b_2 = \frac12, \quad a_{11} = \eta, \quad a_{12} = \eta,
	\quad a_{22} = \frac32\eta,
$$
where $\eta=\eta_{m}:= 10^{-3}$ and $\eta=\eta_{M}:= 1.185186593672589$, 
which corresponds to a lower bound on the Hessian of the 
free energy \eqref{2.H} approximately equal to $10^{-6}$.
In all cases, the initial data have the form
$$ c_{1}^{in}(x) = c_{1,A} + (c_{1,B}-c_{1,A})x^{10},\quad
c_{2}^{in}(x) = c_{2,A} + (c_{2,B}-c_{2,A})x^{1/10}, $$
which describes an accumulation of $c_{1}$, $c_{2}$ close to $x=1$, $x=0$,
respectively. The parameters $c_{i,A}$, $c_{i,B}$, $i=1,2$, 
are chosen in such a way that
$p(c_{1,A},c_{2,A})=p(c_{1,B},c_{2,B})=1$, which is necessary 
in order to have convergence to a steady state in the case of Dirichlet 
boundary conditions, since any steady state
is characterized by the pressure assuming a constant value.  

For homogeneous Neumann boundary conditions and $\eta=\eta_{m}$ (Case I),
Figure \ref{fig1} shows the evolution of the
mass densities $c_1$, $c_2$ and the pressure $p$ at the time instants
$t=0,5\cdot 10^{-3},50\cdot 10^{-3},1$ (the solution at $t=1$ represents the
steady state) as well as the relative free energy $\en(c(t))-\en(c^0)$.
As expected, the pressure converges to a constant function for ``large'' times. The
stationary mass densities are nonconstant. 
The Neumann boundary condition
is numerically satisfied, but we observe a boundary layer at $x=0$,
originating from the ``constraint'' of constant pressure. The relative free energy
decays exponential fast. After $t\approx 0.7$, the stationary state is
almost reached and the values of the free energy are of the order to the
numerical precision.

\begin{figure}[ht]
\hspace*{-10mm}\includegraphics[width=18cm]{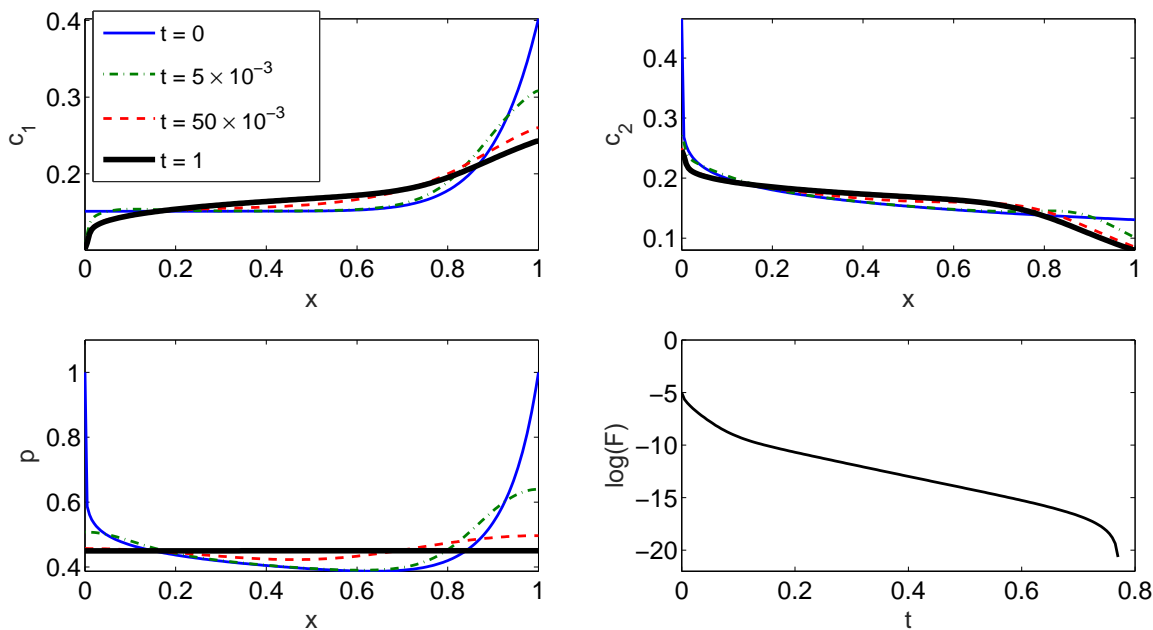}
 \caption{Case I (Neumann conditions, $\eta=\eta_{m}$): evolution of the
mass densities $c_1$, $c_2$, pressure $p$, and logarithm of the
relative free energy $\en(c(t))-\en(c^0)$.}
\label{fig1}
\end{figure}

In Figure \ref{fig2}, we present the results for $\eta=\eta_{M}$
(Case II), still with homogeneous Neumann boundary conditions. 
We observe that the relative free energy decay is slightly
slower than in Case I but still exponential fast.

\begin{figure}[ht]
\hspace*{-10mm}\includegraphics[width=18cm]{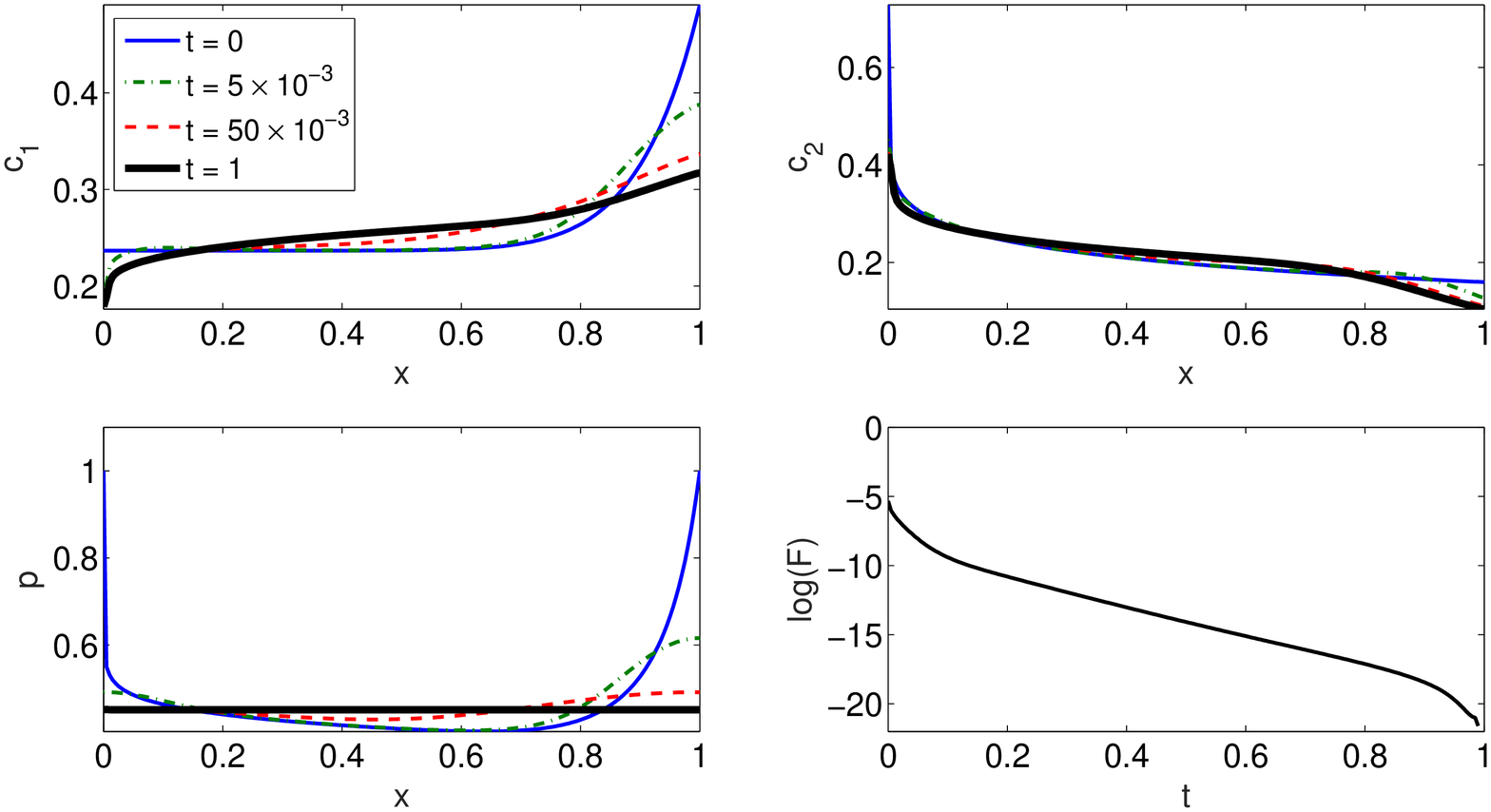}
 \caption{Case II (Neumann conditions, $\eta=\eta_{M}$): evolution of the
mass densities $c_1$, $c_2$, pressure $p$, and logarithm of the
relative free energy $\en(c(t))-\en(c^0)$.}
\label{fig2}
\end{figure}

For the case of Dirichlet boundary conditions, an additional term has to be 
added to the free energy in order to have free energy decay, due to the presence 
of additional boundary contributions in the free energy balance equation. 
More precisely, we choose the modified free energy 
$\widetilde\en(c):=\en(c)-(\alpha_1 c_1+\alpha_2 c_2)$, where 
$\alpha_1$, $\alpha_2\in\R$ are such that the boundary term in
\begin{align*}
  \frac{d}{dt}\int_\Omega\widetilde\en(c)dx
	&= \int_\Omega\sum_{i=1}^2\frac{\pa\en}{\pa c_i}\diver(c_i\na p) dx
	- \int_\Omega\sum_{j=1}^2\alpha_j\diver(c_j\na p) dx \\
	&= -\int_\Omega|\na p|^2 dx 
	+ \int_{\pa\Omega}\bigg(\sum_{i=1}^2c_i\mu_i 
	- \sum_{j=1}^2 \alpha_j c_j\bigg)\na p\cdot\nu ds
\end{align*}
vanishes. Here, we have used the relations $\pa\en/\pa c_i=\mu_i$ and
$\sum_{i=1}^2 c_i\na\mu_i=\na p$ (see \eqref{2.nap}). The boundary term 
vanishes if $(\alpha_1,\alpha_2)$ solves the linear system
$$
  \begin{pmatrix}	c_1^L & c_2^L \\ c_1^R & c_2^R 	\end{pmatrix}
	\begin{pmatrix}	\alpha_1 \\ \alpha_2	\end{pmatrix}
	= \begin{pmatrix} 
	c_1^L\mu_1^L + c_2^L\mu_2^L \\ c_1^R\mu_1^R + c_2^R\mu_2^R \end{pmatrix},
$$
where $c_i^L$, $c_i^R$ are the values of $c_i$ at $x=0$, $x=1$, respectively,
and $\mu_i^{L/R} = \mu_i(c_1^{L/R},c_2^{L/R})$ for $i=1,2$. 
If $c_1^L/c_2^L\neq c_1^R/c_2^R$, the above linear system is uniquely solvable.
We remark that the modified free energy $\widetilde\en$ does not change
the energy dissipation $\int_\Omega|\na p|^2 dx$ but it is nontrivial, as
$\int_\Omega(\alpha_1c_1+\alpha_2c_2)dx$ is nonconstant in time.

Figures \ref{fig3} and \ref{fig4} illustrate the evolution of 
$c_1$, $c_2$, $p$, and of the modified relative free energy
$\widetilde\en(c(t))-\widetilde\en(c^0)$. Again, the mass densities
at $t=1$ (they are basically stationary)
are nonconstant, and the modified relative free energy converges exponentially
fast. The decay rate is faster for $\eta=\eta_{M}$, contrarily to what happens in
the case of Neumann boundary conditions.

\begin{figure}[ht]
\hspace*{-10mm}\includegraphics[width=18cm]{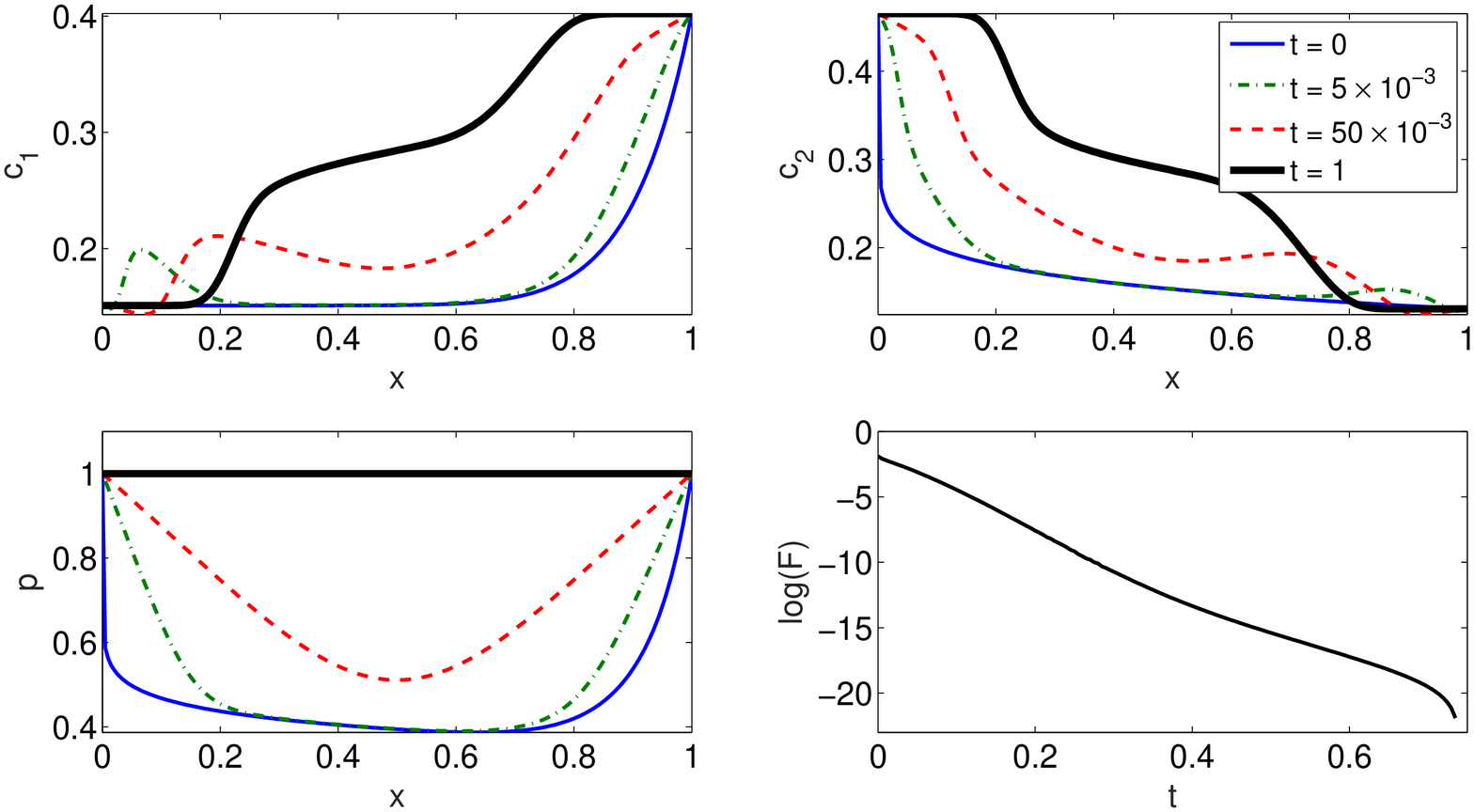}
 \caption{Case III (Dirichlet conditions, $\eta=\eta_{m}$): evolution of the
mass densities $c_1$, $c_2$, pressure $p$, and logarithm of the
relative modified free energy $\widetilde\en(c(t))-\widetilde\en(c^0)$.}
\label{fig3}
\end{figure}

\begin{figure}[ht]
\hspace*{-10mm}\includegraphics[width=18cm]{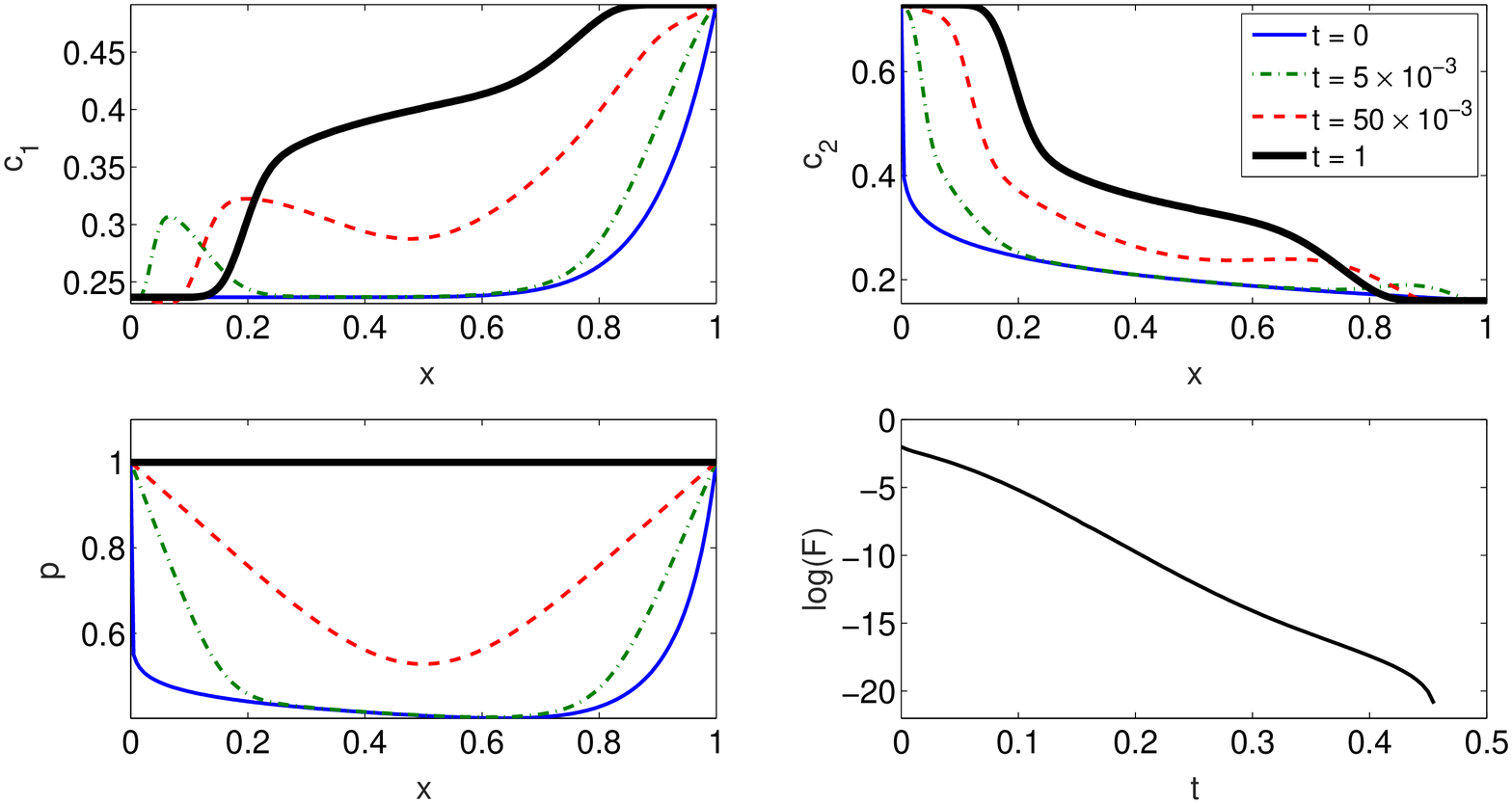}
 \caption{Case IV (Dirichlet conditions, $\eta=\eta_{M}$): evolution of the
mass densities $c_1$, $c_2$, pressure $p$, and logarithm of the
relative modified free energy $\widetilde\en(c(t))-\widetilde\en(c^0)$.}
\label{fig4}
\end{figure}


\begin{appendix}
\section{Formal proof of \eqref{1.ii}}

We prove the integral identity \eqref{1.ii} in a formal setting. 
We proceed as in the proof of Theorem \ref{coro}. Let
$\psi(c)=\ctot f(c_1/\ctot,\ldots,c_{n-1}/\ctot)$ for $c\in\dom$, and 
let $z_{i}=c_{i}/\ctot$.
Since $\eps=0$, the statement follows if 
$\sum_{j=1}^n c_j\pa^2\psi/\pa c_i\pa c_j=0$ for
$i=1,\ldots,n$. A straightforward computation gives
\begin{align*}
  \frac{\pa^2\psi}{\pa c_i\pa c_j} 
	&= \sum_{k=1}^{n-1}\frac{\pa f}{\pa z_{k}}\left( \frac{\pa z_{k}}{\pa c_{i}} 
	+ \frac{\pa z_{k}}{\pa c_{j}} + \ctot\frac{\pa^{2}z_{k}}{\pa c_{i}\pa c_{j}} \right)
  + \ctot\sum_{k,s=1}^{n-1}\frac{\pa^{2} f}{\pa z_{k}\pa z_{s}}
	\frac{\pa z_{k}}{\pa c_{i}}\frac{\pa z_{s}}{\pa c_{j}}.
\end{align*}
Since $\pa z_{k}/\pa c_{i} = \delta_{ik}/\ctot - c_{k}/\ctot^{2}$, it follows that 
$$
  \sum_{i=1}^{n}c_{i}\frac{\pa z_{k}}{\pa c_{i}} = 0,\quad k=1,\ldots,n-1.
$$
Moreover,
$$ 
  \frac{\pa z_{k}}{\pa c_{i}} + \frac{\pa z_{k}}{\pa c_{j}} 
	+ \ctot\frac{\pa^{2}z_{k}}{\pa c_{i}\pa c_{j}} = 0, \quad k=1,\ldots,n-1,\ 
	i,j=1,\ldots,n. 
$$
Putting these three identities together yields
$\sum_{j=1}^nc_j\pa^2\psi/\pa c_i\pa c_j=0$ for $i=1,\ldots,n$.
\end{appendix}



\begin{thebibliography}{11}
\bibitem{ADF85} G.~Acs, S.~Doleschall, and E.~Farkas. General purpose compositional 
model. {\em Soc. Petroleum Engin. J.} 25 (1985), 543-553.

\bibitem{Ama93} H.~Amann. Nonhomogeneous linear and quasilinear elliptic and parabolic 
boundary value problems. In: H.J.~Schmeisser and H.~Triebel (editors), 
{\em Function Spaces, Differential Operators and Nonlinear Analysis}, pages
9–126. Teubner, Stuttgart, 1993.

\bibitem{BBKT03} S. Berres, R.~B\"urger, K.~Karlsen, and E.~Tory. Strongly
degenerate parabolic-hyperbolic systems modeling polydisperse sedimentation
with compression. {\em SIAM J. Appl. Math.} 64 (2003), 41-80.

\bibitem{BoDr15} D.~Bothe and W.~Dreyer. Continuum thermodynamics of chemically 
reacting fluid mixtures. {\em Acta Mech.} 226 (2015), 1757-1805.

\bibitem{BDPS10} M.~Burger, M.~Di Francesco, J.-F.~Pietschmann, and B.~Schlake.
Nonlinear cross-diffusion with size exclusion. {\em SIAM J. Math. Anal.} 42 (2010),
2842--2871.

\bibitem{CHM06} Z.~Chen, G.~Huan, and Y.~Ma. {\em Computational Methods for
Multiphase Flows in Porous Media}. SIAM, Providence, 2006.

\bibitem{CQE00} Z.~Chen, G.~Qin, and R.~Ewing. Analysis of a compositional model
for fluid flow in porous media. {\em SIAM J. Math. Anal.} 60 (2000), 747-777.

\bibitem{DrJu12} M.~Dreher and A.~J\"ungel. Compact families of piecewise constant 
functions in $L^p(0,T;B)$. {\em Nonlin. Anal.} 75 (2012), 3072-3077.


\bibitem{Gao00} D.~Y.~Gao. {\em Duality Principles in Nonconvex Systems.
Theory, Methods and Applications.} Springer, 2000.

\bibitem{GMWZ06} C.~Gu\`es, G.~M\'etivier, M.~Williams, and K.~Zumbrun. 
Navier-Stokes regularization of multidimensional Euler shocks.
{\em Ann. Sci. Ecole Normale Sup.} 39 (2006), 75-175. 

\bibitem{HoFi05} H.~Hoteit and A.~Firoozabadi. Multicomponent fluid flow by 
discontinuous Galerkin and mixed methods in unfractured and fractured media. 
{\em Water Resources Research} 41 (2005), W11412, 15 pages.

\bibitem{Joh14} D.~Johnston. {\em Advances in Thermodynamics of the
van der Waals Fluid.} Morgan \& Claypool Publishers, USA, 2014.

\bibitem{Jue15} A.~J\"ungel. The boundedness-by-entropy method for cross-diffusion
systems. {\em Nonlinearity} 28 (2015), 1963-2001.

\bibitem{LMM11} A.~Lamorgese, D.~Molin, and R.~Mauri. Phase field approach to
multiphase flow modeling. {\em Milan J. Math.} 79 (2011), 597-642.

\bibitem{Lec11} M.~L\'ecureux-Mercier. Global smooth solutions of Euler equations 
for van der Waals gases. {\em SIAM J. Math. Anal.} 43 (2011), 877-903.

\bibitem{MiPo16} J.~Miky\v{s}ka and O. ~Pol\'\i vka. Energy inequalities in
compositional simulation. {\em Proceedings of ALGORITMY 2016}
(2016), 224-233.

\bibitem{OlPa07} S.~Oladyshkin and M.~Panfilov. Limit thermodynamic model for 
compositional gas-liquid systems moving in a porous medium. 
{\em Transp. Porous Media} 70 (2007), 147-165. 

\bibitem{PeRo76} D.~Peng and D.~Robinson. A new two-constant equation of state. 
{\em Industrial Engin. Chem.: Fundamentals} 15 (1976), 59-64.

\bibitem{PoMi15} O. ~Pol\'\i vka and J.~Miky\v{s}ka. Compositional modeling of 
two-phase flow in porous media using semi-implicit scheme. 
{\em IAENG Intern. J. Appl. Math.} 45 (2015), 218-226.

\bibitem{Soa72} G.~Soave. Equilibrium constants from a modified Redlich-Kwong 
equation of state. {\em Chem. Engin. Sci.} 27 (1972), 1197-1203.

\bibitem{THHC14} M.~Duc Thanh, N.~Dinh Huy, N.~Huu Hiep, and D.~Huy Cuong.
Existence of traveling waves in van der Waals fluids with viscosity and 
capillarity effects. {\em Nonlin. Anal.} 95 (2014), 743-755.

\bibitem{Whi86} S.~Whitaker. Flow in porous media I: A theoretical derivation of 
Darcy's law. {\em Transp. Porous Media} 1 (1986), 3-25.

\bibitem{ZaJu16} N.~Zamponi and A.~J\"ungel. Analysis of degenerate cross-diffusion 
population models with volume filling. To appear in 
{\em Ann. Inst. H. Poincar\'e AN}, 2016. {\tt arXiv:1502.05617}.

\bibitem{Zei90} E.~Zeidler. {\em Nonlinear Functional Analysis and its Applications.}
Volume II/B. Springer, New York, 1990.

\bibitem{Zha13} S.-Y.~Zhang. Existence of multidimensional non-isothermal phase
transitions in a steady van der Waals flow. {\em Discrete Cont. Dyn. Sys.}
33 (2013), 2221-2239.
\end{thebibliography}
\end{document}